\documentclass[10pt,a4paper]{amsart}

\usepackage[utf8]{inputenc}
\usepackage[oldstylemath,fulloldstylenums]{kpfonts}
\usepackage{xcolor}
\definecolor{color1}{RGB}{83, 156, 37}
\usepackage[pagebackref,colorlinks,linkcolor=color1,citecolor=color1,urlcolor=color1,hypertexnames=false]{hyperref}
\usepackage{enumerate}

\usepackage{amsmath,amsthm,amssymb,amsrefs}
\usepackage{tikz-cd}

\newcommand{\comment}[1]{} 
\newcommand{\xto}{\xrightarrow}	
\def\N{\mathbb{N}}

\def\R{\mathbb{R}} 
\def\C{\mathbb{C}}

\newcommand{\defeq}{\mathrel{:=}} 

\newcommand{\exit}{\mathrm{exit}}
\newcommand{\Exit}{\mathrm{Exit}}

\DeclareMathOperator{\Psh}{Psh}
\DeclareMathOperator{\Shv}{Shv}
\newcommand{\Sp}{\mathrm{Sp}}
\newcommand{\An}{\mathrm{An}}

\newcommand{\cF}{\mathcal{F}}
\newcommand{\cG}{\mathcal{G}}
\newcommand{\cO}{\mathcal{O}}
\newcommand{\cK}{\mathcal{K}}
\newcommand{\cKO}{\mathcal{KO}}
\newcommand{\uE}{\underline{E}}
\newcommand{\preuE}{\uE^{\pre}}
\newcommand{\preutE}{\widetilde{\uE}^{\pre}}
\newcommand{\colim}{\operatornamewithlimits{colim}}
\newcommand{\uSing}{\underline{\mathrm{Sin}}\mathrm{g}}

\DeclareMathOperator{\id}{id}
\DeclareMathOperator{\E}{E}
\DeclareMathOperator{\Ind}{Ind}

\DeclareMathOperator{\NcMot}{NcMot}
\DeclareMathOperator{\pre}{pre}

\DeclareMathOperator{\Orbit}{Orb}
\newcommand{\Orb}{\Orbit_G}
\DeclareMathOperator{\Mod}{Mod}
\DeclareMathOperator{\ho}{ho}
\DeclareMathOperator{\Cat}{Cat}
\DeclareMathOperator{\dual}{dual}
\DeclareMathOperator{\Fun}{Fun}
\DeclareMathOperator{\Sing}{Sing}
\DeclareMathOperator{\Set}{Set}

\DeclareMathOperator{\CHaus}{CHaus}

\DeclareMathOperator{\Topp}{Top}
\DeclareMathOperator{\LCHaus}{LCHaus}
\DeclareMathOperator{\LCHopen}{LCHaus_{open}}
\DeclareMathOperator{\LCHp}{LCHaus_{proper}}

\newcommand{\op}{\mathrm{op}}

\DeclareMathOperator{\stable}{st}
\DeclareMathOperator{\Lan}{Lan}
\DeclareMathOperator{\Ran}{Ran}
\DeclareMathOperator{\sheafify}{sh}
\newcommand{\sh}{{\sheafify}}
\DeclareMathOperator{\Ctbl}{\Shv_c}

\DeclareMathOperator{\fib}{fib}

\DeclareMathOperator{\res}{res}
\DeclareMathOperator{\pr}{Pr}

\DeclareMathOperator{\sep}{sep}
\DeclareMathOperator{\GCsAlg}{G-C*Alg}
\DeclareMathOperator{\CsAlg}{C*Alg}
\newcommand{\prr}{\pr^R}
\newcommand{\prl}{\pr^L}
\newcommand{\prll}{\pr^{LL}}

\newcommand{\GTop}{\Topp_G}
\newcommand{\GCH}{\CHaus_G}
\newcommand{\GLCH}{\LCHaus_G}

\newcommand{\bre}{\Gamma^G_{\mathrm{Br}}}
\newcommand{\brec}{\Gamma^G_{\mathrm{Br,c}}}
\newcommand{\breg}{\underline{\Gamma}_{\mathrm{Br}}}

\newcommand{\Catdual}{\Cat_{\dual}}

\def\topdf{\texorpdfstring}

\numberwithin{equation}{section}

\theoremstyle{plain}
\newtheorem{thm}[equation]{Theorem}
\newtheorem*{thm*}{Theorem}

\newtheorem{lem}[equation]{Lemma}
\newtheorem{coro}[equation]{Corollary} 
\newtheorem{prop}[equation]{Proposition}

\newtheorem{introthm}{Theorem}

\theoremstyle{definition}

\newtheorem{defn}[equation]{Definition} 
\newtheorem*{defn*}{Definition}
\newtheorem{ex}[equation]{Example}
\newtheorem{exs}[equation]{Examples}
\newtheorem{rmk}[equation]{Remark}

\theoremstyle{remark} 
\newtheorem*{ack}{Acknowledgements}

\author{Guido Arnone}
\email{garnone@dm.uba.ar}
\address{Departamento de Matem\'atica/IMAS\\ Facultad de Ciencias Exactas y Naturales\\ Universidad de Buenos Aires\\ Ciudad Universitaria \\ (1428) Buenos Aires}

\author{Devarshi Mukherjee}
\email{devarshi.mukherjee@maths.ox.ac.uk}
\address{Mathematical Institute\\ University of Oxford\\ Woodstock Rd\\ OX26GG Oxford}

\author{Thomas Nikolaus}
\email{nikolaus@uni-muenster.de}
\address{University of M\"unster\\ Mathematics M\"unster\\ Einsteinstrasse 62\\ 48149 M\"unster}


\title{Bredon sheaf cohomology}

\begin{document}

\begin{abstract}
For a finite group $G$, we compute the algebraic $K$-theory of the category of equivariant sheaves on a locally compact Hausdorff $G$-space, generalizing a result of Efimov, and determine the equivariant $E$-theory of the $C^*$-algebra of continuous functions. These invariants admit natural descriptions in terms of a new equivariant cohomology theory, which we call Bredon sheaf cohomology.

This theory recovers classical Bredon cohomology for $G$-CW complexes and ordinary sheaf cohomology when $G$ is trivial. We establish its basic structural properties and prove a strong uniqueness theorem: any functor from the category of locally compact Hausdorff $G$-spaces to a dualizable stable category satisfying equivariant open descent and cofiltered compact codescent is equivalent to Bredon sheaf cohomology, generalizing a result of Clausen.
	\end{abstract}
	
	\maketitle
	\setcounter{tocdepth}{1} 
	\tableofcontents

\section{Introduction}

Let $X$ be a locally compact Hausdorff space and let
$\Shv(X,\mathcal{C})$ denote the $\infty$-category of sheaves on $X$
with values in a presentable, stable $\infty$-category $\mathcal{C}$.
If $\mathcal{C}$ is dualizable (\cite{sag}*{D.7.3}), then so is $\Shv(X,\mathcal{C})$,
and its algebraic $K$-theory is therefore defined (\cite{efiloc}). 
A fundamental result of Efimov identifies this $K$-theory in purely
geometric terms.

\begin{thm*}[Efimov, \cite{efiloc}*{Theorem 0.2}]\label{thm_efimov}
There is a natural equivalence
\[
K(\Shv(X,\mathcal{C})) \simeq \Gamma_c(X, {K\mathcal{C}}),
\]
where the right-hand side denotes compactly supported sheaf cohomology
of $X$ with values in the spectrum $K\mathcal{C}$.
\end{thm*}

Efimov’s theorem provides a powerful bridge between algebraic $K$-theory
and geometric topology, and has sparked significant recent interest.
For instance, Lehner has established a generalization to stably locally compact spaces \cite{Lehner}. 
Among other applications, Efimov’s result yields categorical models for
assembly maps of the form
\[
H_*(M, K\mathbb{Z}) \longrightarrow K_*(\mathbb{Z}[\pi_1 M])
\]
for compact manifolds $M$ (and more generally compact ANRs).
For aspherical $M$, this map is conjectured to be an equivalence; this is
a special case of the Farrell--Jones conjecture and would imply the
Borel conjecture via surgery theory.

To treat assembly maps in the full generality predicted by the
Farrell--Jones conjecture, one is naturally led to seek equivariant
refinements.
Let $G$ be a finite group acting on $X$.
We define the $\infty$-category of $G$-equivariant sheaves by
\[
\Shv_G(X,\mathcal{C}) := \Shv(X,\mathcal{C})^{hG}.
\]
This category is again dualizable (see Remark \ref{rmk:shvg=shvg}), and hence its
algebraic $K$-theory is defined.
The first main result of this paper is the following equivariant
analogue of Theorem~\ref{thm_efimov}.

\begin{introthm}[Theorem \ref{thm:shvmot}]\label{thm_zwei}
There is a natural equivalence
\[
K(\Shv_G(X,\mathcal{C})) \simeq \brec(X, {K_G\mathcal{C}}),
\]
where the right-hand side denotes compactly supported \emph{Bredon
sheaf cohomology} of $X$ with values in (a version of) the $G$-equivariant algebraic $K$-theory spectrum $K_G\mathcal{C}$.
\end{introthm}

The cohomology theory appearing on the right-hand side is new.
To the best of our knowledge, a sheaf-theoretic refinement of Bredon cohomology
has not previously been constructed.
The closest related work is due to Honkasalo \cites{honkasalo, honkasalo2},
who develops a theory for ordinary abelian coefficients.

The primary goal of this paper is to define and develop this
\emph{Bredon sheaf cohomology}, which may be viewed as a synthesis of
classical sheaf cohomology and Bredon cohomology.
It interpolates between the two theories:
for trivial $G$ it recovers ordinary sheaf cohomology, while for $G$--CW complexes it recovers classical (singular) Bredon cohomology. 

\subsection*{Bredon Sheaf Cohomology}

The input data for the theory is a coefficient system, namely a functor
\[
E \colon \Orb^{\op} \longrightarrow \Sp,
\]
where $\Orb$ denotes the orbit category of $G$, whose objects are the
transitive $G$--sets $G/H$ for subgroups $H \subseteq G$.
For example, the coefficient system relevant to
Theorem~\ref{thm_zwei} is given  by 
\[
K_G \mathcal{C} \colon G/H \longmapsto K\bigl(\Fun(BH,\mathcal{C})\bigr).
\]
This is a form of equivariant algebraic $K$-theory (see e.g. \cite{MaximeKaif} where it is the coBorel theory).

We equip the category $\GTop$ of topological spaces equipped with a $G$-action with a Grothendieck topology by declaring coverings to
consist of $G$--invariant open covers.
This endows $\GTop$ with the structure of a (large) site. The functor
\[
t \colon \Orb \longrightarrow \GTop,
\]
sending an orbit to the corresponding discrete $G$--space, is a
morphism of sites when $\Orb$ is endowed with the trivial topology.
Consequently, any coefficient system
$E \in \Fun(\Orb^{\op},\Sp)$ determines a sheaf
$t^*(E) \in \Shv(\GTop,\Sp)$. 

\begin{defn*}
Let $X$ be a $G$--space, and $E \colon \Orb^{\op} \to \Sp$ a
coefficient system.
The \emph{Bredon sheaf cohomology} of $X$ with coefficients in $E$ is
defined by
\[
\bre(X,E) := t^*(E)(X).
\]
\end{defn*}

Unwinding the definition, the functor $t^*$ is given by left Kan extension 
followed by sheafification.
Since $G$--invariant open subsets of $X$ are in natural bijection with
open subsets of the orbit space $X/G$ via the quotient map
$q \colon X \to X/G$, this yields a concrete description.
The value $\bre(X,E)$ is computed as the global sections
of a sheaf
\[
\underline{E}_X \in \Shv(X/G,\Sp),
\]
obtained by sheafifying the presheaf
\[
U \longmapsto \colim_{\,q^{-1}(U)\to Z} E(Z),
\]
where the colimit ranges over all $G$--equivariant maps from
$q^{-1}(U)$ to orbits $Z \in \Orb$.

Conceptually, the sheaf $\underline{E}_X$ reflects the orbit-type geometry
of the $G$--space $X$.
Its stalks record the values of the coefficient system on stabilizers,
and its variation is controlled by how orbit types specialize in the
quotient $X/G$. This makes 
Bredon sheaf cohomology amenable to explicit geometric computations.

\subsection*{Structural Properties and Uniqueness}

Bredon sheaf cohomology has a number of fundamental properties which we prove: 

\begin{enumerate} 
\item \emph{Normalization}: there is a natural equivalence $\bre(Z, E) \simeq E(Z)$ for $Z \in \Orb$.
\item \emph{Open descent}: $\bre(-,E)$ is a sheaf on $\GTop$. 
\item \emph{Cofiltered compact codescent}: for a cofiltered limit of compact Hausdorff $G$--spaces $X = \lim_i X_i$, the map
$\colim_i \bre(X_i,E) \longrightarrow \bre(X,E)$ is an equivalence. 
\item \emph{$G$--homotopy invariance}: every $G$--homotopy equivalence $X \to Y$ between locally compact Hausdorff spaces induces an equivalence $\bre(Y,E) \simeq \bre(X,E). $
\item \emph{Agreement with singular Bredon cohomology}: if $X$ is a sufficiently nice space, e.g. a $G$--CW complex, then Bredon sheaf cohomology agrees with singular Bredon cohomology.
\end{enumerate}

Properties (1) and (2) are true by definition.
Property (3), proven in Theorem \ref{thm:breprofin}, is the main technical result of the paper and relies on a
detailed analysis of the sheaves $\underline{E}_X$, making essential
use of the existence of slices for group actions, as guaranteed by
Abels’ theorem (\cite{abels}*{Theorem 3.3}).
Properties (4) and (5) follow from standard homotopical 
arguments; see Proposition \ref{prop:htpyinv} and Section \ref{subsec:bredoncw}.

In addition, the theory of Bredon sheaf cohomology admits an interpretation in terms of the $G$-shape.
For a general $G$-space $X$, we construct a pro-$G$-anima $\underline{\Pi}_\infty(X)$ such that $\bre(X,E)$ agrees with the singular Bredon cohomology of $\underline{\Pi}_\infty(X)$; see Section \ref{sec:shape}.

The central structural result of the paper is a strong uniqueness theorem
for Bredon sheaf cohomology on $\GLCH$, the category of locally compact
Hausdorff $G$-spaces.
It may be viewed as a $G$-equivariant refinement of a theorem of Clausen
\cite{som}*{Theorem 3.6.11}, building on ideas of Efimov and Hoyois. Let 
$\Fun^{(2),(3)}(\GLCH^{\op},\Sp)$ denote the full subcategory of functors satisfying  properties (2) and (3) above namely open descent and cofiltered compact codescent.

\begin{introthm}[Theorem \ref{thm:breeq}]\label{thm_unique}
Restriction to orbits induces an equivalence
\[
\Fun^{(2),(3)}(\GLCH^{\op},\Sp)
\;\simeq\;
\Fun(\Orb^{\op},\Sp),
\]
with inverse given by Bredon sheaf cohomology.
\end{introthm}

In particular, $G$-homotopy invariance and compatibility with singular
Bredon cohomology are formal consequences of open descent and
cofiltered compact codescent alone.
This is somewhat surprising, as it excludes the existence of
non--homotopy-invariant theories satisfying these axioms.
The theorem remains valid with values in any compactly assembled
$\infty$-category in place of $\Sp$.

\subsection*{Compact Supports and Applications}

For a locally compact Hausdorff $G$-space $X$ we define compactly supported Bredon sheaf cohomology by 
\[
\brec(X,E) := \Gamma_c(X/G,\underline{E}_X).
\]
For compact $X$, this agrees with ordinary Bredon sheaf cohomology.
In general, it satisfies open codescent, cofiltered compact codescent,
open--closed excision, and proper $G$--homotopy invariance (see Section~\ref{subsec:csupp}).

We also have a variant of Theorem \ref{thm_unique} for compactly supported  Bredon sheaf cohomology (Theorem \ref{thm:uniqueness}).
This then directly implies Theorem~\ref{thm_zwei}: basic properties of $K$--theory
and sheaf categories ensure that
$K(\Shv_G(X,\mathcal{C}))$ satisfies the defining axioms and therefore
coincides with compactly supported Bredon sheaf cohomology.

The same formalism applies in greater generality.
In particular, applying the uniqueness theorem to equivariant topological
$K$-theory yields the following identification of equivariant topological $K$-theory.  This is classically defined for a Hausdorff $G$-space using equivariant vector bundles, see \cite{segal-eqk}, but is equivalent to topological $K$-theory of the $C^*$-algebraic crossed product $G \ltimes C(X)$. 
\begin{introthm}[Corollary \ref{coro:K-theory-crossed}]\label{thm_zwei'}
Let $X$ be a locally compact Hausdorff $G$-space. There is a natural equivalence
\[K^{\mathrm{top}}(G \ltimes C_0(X)) \simeq \brec(X,K^{\mathrm{top}}_G),
\]
where $K_G$ denotes the restriction of $K^{\mathrm{top}}(G \ltimes C_0(-))$
to the orbit category, which sends $G/H$ to $K^{\mathrm{top}}_G(G/H) = K^{\mathrm{top}}_H(\mathrm{pt})$. 
\end{introthm}

This result provides a topological counterpart to
Theorem~\ref{thm_zwei}.
More precisely, under the analogy between dualizable categories
and $C^\ast$-algebras—where $G$-equi\-variant sheaf categories correspond
to crossed product $C^\ast$-algebras and algebraic $K$-theory corresponds
to topological $K$-theory—the two theorems are parallel.

Similar arguments yield further refinements, including an
identification of the noncommutative motive of
$\Shv_G(X,\mathcal{C})$ and, in the presence of a $G$--action on $X$,
an identification of the $G$--motive of $\Shv(X,\mathcal{C})$ itself as well as an identification of $C_0(X)$ as an object in the equivariant $E$-theory.
These applications are carried out in Section~\ref{sec:appl}.

\subsection*{Geometric Description and Computability} 

The intuition above can be made precise. 
The sheaf $\underline{E}_X$ admits a concrete geometric description which
both explains the formal properties of Bredon sheaf cohomology and enables
explicit calculations in practice. 

\begin{introthm}[Theorems \ref{thm:EX-constr} and \ref{thm:mayeda}]\label{prop_constructible}
Let $X$ be a Tychonoff $G$--space, and $E \colon \Orb^{\op} \to \Sp$ a
coefficient system.
\begin{enumerate}
\item The stalk of $\underline{E}_X$ at a point $x \in X/G$ corresponding
to an orbit $G/H$ is canonically equivalent to $E(G/H)$.

\item If $X$ is locally compact Hausdorff, then the sheaf $\underline{E}_X$ is constructible with respect to the
stratification of $X/G$ by orbit types.

\item If $X$ is a $G$--manifold, then
$\underline{E}_X$ is classified by the composite
\[
\Exit(X/G) \longrightarrow \Orb^{\op} \xrightarrow{E} \Sp,
\]
where $\Exit$ denotes the exit-path category and the first
functor arises as the straightening of the right fibration $\Exit(X) \to  \Exit(X/G)$, see \cite{mayeda}.
\end{enumerate}
\end{introthm}

In particular, the last part reduces the computation of
Bredon cohomology to a limit over $\Exit(X/G)$,
and thus to calculations
on strata and their incidence data, much as in classical constructible
sheaf theory. In particular, for $G$--manifolds or spaces with finitely many orbit types,
this description reduces Bredon sheaf cohomology to explicit calculations
on strata and their incidence relations. 

\begin{ack}
We thank Ulrich Bunke, Benjamin Dünzinger, Thorger Geiß, Janou Glaeser, Achim Krause, Markus Land, Phil Pützstück, and Maxime Ramzi for many helpful discussions related to this work. We are particularly grateful to Markus Land and Maxime Ramzi for generously sharing their insights and ideas on the proof of the uniqueness result in the non-equivariant case. We also point to forthcoming work by Valerio Proietti and Makoto Yamashita, who have related results in the setting of \'etale groupoid \(C^*\)-algebras and equivariant \(\mathrm{KK}\)-theory, and thank them for sharing parts of their work.

The first named author was supported by a CONICET postdoctoral fellowship and partially supported by grants UBACyT 206BA, PICT 710 and Mathematics Münster's ``Young Research Fellows'' visitors program. He wishes to express his gratitude to the Department of Mathematics at Universität Münster for their hospitality during his visit, where part of the research for this project was carried out. 

The second named author was supported by a DFG Eigenestelle (project number 534946574) and a UK Research and Innovation Horizon Europe Guarantee MSCA Postdoctoral Fellowship. 

All authors were funded by the Deutsche Forschungsgemeinschaft (DFG, German Research Foundation) – Project-ID 427320536 – SFB 1442, as well as under Germany’s Excellence Strategy EXC2044/2–390685587, Mathematics Münster: Dynamics–Geometry–Structure.
\end{ack}

\section{Preliminaries}

	Throughout the article we fix a finite group $G$. We shall write $\GLCH$, $\GCH$ and $\Orb$ for the (1-)categories of locally compact Hausdorff, compact Hausdorff, and (discrete) transitive $G$-spaces respectively, and $\GTop$ for the category of all $G$-spaces. We shall freely use the language of $\infty$-categories as developed in \cites{htt,ha}, and refer to them simply as categories.
	
	\subsection{Presentable and dualizable categories}
	
	Recall that a category is \emph{presentable} if it is cocomplete and $\kappa$-compactly generated for some regular cardinal $\kappa$. When $\kappa = \omega$ we omit it from the notation.
	We write $\Cat_\infty$ for the category of not necessarily small categories, $\Pr$ for the subcategory spanned by presentable categories and $\prl$ (resp. $\prr$) for the subcategories of $\pr$ spanned by presentable categories together with left (resp. right) adjoints. 
	A category is \emph{compactly assembled} if it is
	a retract in $\prl$ of a compactly generated category.
	
	We put $\prll$ for the subcategory of $\prl$ spanned by \emph{strongly continuous functors}; that is, left adjoint functors whose right adjoint admits a further right adjoint. Recall also that a category is said to be \emph{stable} if it has finite limits and colimits and pullback squares coincide with pushout squares. For any subcategory $C$ of $\Cat_\infty$, we write $C_{\stable}$ for the subcategory of $C$ generated by those categories which are stable.

	\subsection{The Lurie tensor product and dualizable categories}
	
	A functor $C \times D \to E$ between presentable categories is said to be \emph{bilinear}
	if it preserves colimits in each variable separately. The \emph{(Lurie) tensor product} of $C$ and $D$, 
	introduced originally in \cite{ha}*{Section 4.8.1},
	is a presentable category $C \otimes D$ equipped with a bilinear functor $C \times D \to C \otimes D$
	such that for each $E \in \prl_{\stable}$ the map
	\[
	\Fun^L(C \otimes D, E) \to \Fun^{\mathrm{biL}}(C \times D, E) = \Fun^L(C,\Fun^L(D,E))
	\]
	is an equivalence. The bifunctor $\otimes \colon \prl \times \prl \to \prl$ is continuous in each variable and it restricts to a bifunctor $\prl_{\stable} \times \prl_{\stable} \to \prl_{\stable}$, and 
	we have $C \cong \Sp \otimes C \cong C \otimes \Sp$ for all $C \in \prl_{\stable}$. Furthermore $\otimes$ promotes to a symmetric monoidal category structure on $\prl_{\stable}$.
	A detailed treatment of the Lurie tensor product 
	can be consulted in \cite{som}*{Section 2.8}. 
	
	A stable category is \emph{dualizable} if it is a dualizable object of $\prl_{\stable}$ with respect to $\otimes$. Equivalently, a category is dualizable if it is stable and compactly assembled (\cite{som}*{Theorem 2.9.2}). We put $\Catdual$ for the subcategory of $\prll_{\stable}$ generated by dualizable categories.
	
	\begin{exs} The categories $\Sp$ of spectra, and more generally $\Mod(R)$ of modules over a given ring spectrum $R$ are dualizable. Another example is that of the derived category $D(R)$ of a ring $R$.
	\end{exs}
	
	We refer to \cite{efiloc} and \cite{som} for a comprehensive treatment of dualizable categories. 
	
	\subsection{Sheaves and $k$-sheaves}
	
	For a given $X \in \GLCH$ and subspaces $A, B \subset X$, we write $A \Subset B$ if $A=B$ or there exists a $G$-invariant open subspace $U \subset X$ such that $A \subset U  \subset B$. This yields an order relation on the poset $\mathcal P_G(X)$ of $G$-invariant subspaces of $X$. 
	We write $\cO_G(X)$ and $\cK_G(X)$ 
	for the subposets of $G$-invariant open and compact subspaces of $X$ respectively, and $\cKO_G(X)$ for their union. 	
	By setting $G = 1$ we recover the non-equivariant, classical definitions of the posets of open and compact subspaces; to refer to the latter we will simply drop the group from the notation.
	
	\begin{defn} \label{defn:shvx}
		Let $C$ be a complete category and $X$ a topological space.
		The category $\Shv(X,C)$ of $C$-\emph{valued sheaves on $X$} is the subcategory of $\Psh(X,C) = \Fun(\cO(X)^{\op},C)$ generated by functors $F$ satisfying the following conditions:
		\begin{enumerate}[(i)]
			\item $F(\emptyset) \cong \ast$;
			\item for each $U,V \subset X$ open, the square 
			\[
			\begin{tikzcd}
				F(U \cup V) \arrow{r} \arrow{d} & \arrow{d} F(U)\\
				F(V) \arrow{r} & F(U \cap V)
			\end{tikzcd}
			\]
			is a pullback;
			\item for each filtering union of open sets $U = \bigcup_{i \in I} U_i$, the canonical map
			\[
			F(U) \to \lim_{i \in I} F(U_i)
			\]
			is an equivalence.
		\end{enumerate}
		We denote the left adjoint to the inclusion $\Shv(X,C) \hookrightarrow \Psh(X,C)$ by $(-)^\sh$.
		\end{defn}
		
		\begin{defn}[\cite{htt}*{Definition 7.3.4.1}] The category $\Shv_{\mathcal K}(X,C)$ of \emph{$C$-valued $k$-sheaves on $X$} is the subcategory of $\Fun(\cK(X)^{\op},C)$ generated by functors $F$ satisfying the following conditions:
		\begin{enumerate}[(i)]
			\item $F(\emptyset) \cong \ast$;
			\item for each $K,L \subset X$ compact, the square 
			\[
			\begin{tikzcd}
				F(K \cup L) \arrow{r} \arrow{d} & \arrow{d} F(L)\\
				F(K) \arrow{r} & F(K \cap L)
			\end{tikzcd}
			\]
			is a pullback;
			\item for each compact $K \subset X$, the canonical map
			\[
			\colim_{K \Subset L}F(L) \to F(K)
			\]
			is an equivalence.
		\end{enumerate}
	\end{defn}
	
	\begin{rmk} \label{rmk:OG-OX/G}
		Since $G$ is finite, there are natural equivalences between $\cO_G(X)$, $\cK_G(X)$ and $\cKO_G(X)$ and the posets 
		$\cO(X/G)$, $\cK(X/G)$ and $\cKO(X/G)$ on the orbit space $X/G$ of $X$. From here we can make sense of sheaves defined on $\cO_G(X)^\op$ and $k$-sheaves defined on $\cK_G(X)^\op$ canonically.
	\end{rmk}

	\begin{prop}[\cite{som}*{Corollary 2.12.3}] \label{prop:shvx-dual}
		If $X$ is a locally compact Hausdorff space and $D$ is a dualizable category, then $\Shv(X,D)$ is a dualizable category.
		\qed
	\end{prop}
	
	\begin{thm}[\cite{htt}*{Theorem 7.3.4.9}] \label{thm:shv=kshv} Let $X$ be a locally compact Hausdorff space and $C$ a presentable category where filtered colimits are left exact. There are inverse equivalences between sheaves and $k$-sheaves
		\[
		\psi \colon \Shv(X,C) \longleftrightarrow \Shv_{\mathcal K}(X,C) \colon \phi
		\]
		with objectwise formulas given by
		\[
		\psi(\cF)(K) = \colim_{K \subset U} \cF(U), \quad 
		\phi(\mathcal G)(U) = \lim_{K \subset U} \mathcal G(K).
		\]
		\qed
	\end{thm}
	
	Recall that for any stable category $C$, a continuous map $f \colon X \to Y$ yields four different functors
	\begin{equation*}\label{adj:4fun}
		\begin{tikzcd}
			\Shv(Y,C) \arrow[bend left=15]{rr}{f^\ast}
			& \rotatebox{90}{$\vdash$}
			& \Shv(X,C); \arrow[bend left=15]{ll}{f_\ast}
		\end{tikzcd} 
		\begin{tikzcd}
			\Shv(X,C) \arrow[bend left=15]{rr}{f_!}
			& \rotatebox{90}{$\vdash$}
			& \Shv(Y,C) \arrow[bend left=15]{ll}{f^!}
		\end{tikzcd}
	\end{equation*}
	which, as depicted, assemble into adjunctions $f^\ast \dashv f_\ast$ and $f_! \dashv f^!$. There are concrete formulas
	\begin{align*}
		f_\ast \cF(U) &= \cF(f^{-1}(U)); \qquad
		f^\ast \mathcal H(U) = \Big(\colim_{f(U)\subset V \text{ open}} \mathcal H(V)\Big)^\sh;\\
		f_! \cF(U) &= \colim_{f^{-1}(U) \, \supset \, K \to U\text{ proper}} \fib(\cF(f^{-1}(U)) \to \cF(f^{-1}(U \setminus K))).
	\end{align*}
	If $f$ is proper, then $f_\ast = f_!$; if $f$ is an open embedding or more generally a local homeomorphism, then $f^! = f^\ast$. For the unique function $t_X \colon X \to \ast$,  the global sections and compactly-supported global section functors are defined as
	\begin{align*}
		\Gamma \defeq (t_X)_\ast &\colon \Shv(X,C) \to C, \qquad \Gamma_c = (t_X)_! \colon \Shv(X,C) \to C.
	\end{align*}
	For each $x \in X$ and associated map $i_x \colon \ast \to X$, the \emph{stalk} of a sheaf $\cF\in \Shv(X,C)$ at $x \in X$ is defined as $i_x^\ast \cF\in  C$. 
	There is also a constant sheaf functor $\underline{(-)} \defeq (t_X)^\ast \colon C \to \Shv(X,C)$. A sheaf $\cF \in \Shv(X,C)$ is \emph{constant} if it lies in the essential image of $\underline{(-)}$.

	We refer to \cite{Scholze6}*{Lecture VII} and \cite{volpe} for a treatment of $6$-functor formalisms for topological spaces.
	
	To conclude this section, recall that a sequence 
	\[
	A \xto{i} B \xto{p} C
	\]
	in $\Catdual$ is a \emph{Verdier sequence} if it is a fiber-cofiber sequence in $(\Cat_{\infty})_{\stable}$. By \cite{ramzidual}*{Proposition A.20} this is equivalent to requiring $i$ to be fully faithful
	and $p$ to be the cofibre of $i$. 
	
	\begin{prop}[\cite{som}*{}] \label{prop:Shv-cofibre} 
	Let $D$ be a dualizable category. 
	For each locally compact Hausdorff space $X$ and open subspace $U \subset X$, the inclusions $i \colon U \subset X$ and $j \colon X \setminus U \to X$ assemble into a Verdier sequence
	\[
		\Shv(U,D) \xto{i_!} \Shv(X,D) \xto{j^\ast} \Shv(X \setminus U, D).
	\]
	\qed
	\end{prop}
	
	\section{Generalities on \topdf{$G$}{G}-spaces}
	
	In this section, we record
	some general results on locally compact Hausdorff $G$-spaces
	that will be of use throughout the article. The reader may 
	want to postpone this section in a first read, 
	consulting the results as they are referenced.	
	
	\subsection{Trivially proper neighbourhoods}
	
	\begin{defn}[cf. \cite{abels}*{Definition 3.4}] Let $X$ be a locally compact Hausdorff, proper $G$-space. A \emph{trivially proper neighbourhood} of $x \in X$ is a pair $(U,f)$ consisting of a $G$-invariant open subspace $U \ni x$ and a $G$-equivariant continuous map $f \colon U \to G/G_x$ mapping $x$ to $\overline 1 = G/G_x$.   
	\end{defn}
	
	\begin{rmk} A trivially proper neighbourhood of $x \in X$ can 
	equivalently be described as an open, invariant neighbourhood 
	$U \ni x$ together with a retraction $f \colon U \to G\cdot x$ of the inclusion $G\cdot x \subset U$. In particular if $(U,f)$ is a trivially proper neighbourhood
	of $x$ then so is $(V,f|_V)$ for all $G$-invariant opens $x \in V \subset U$.
	\end{rmk}
	
	\begin{rmk}
	Since $G$ is always assumed to be finite, we remark that all $G$-spaces considered are proper.
	\end{rmk}
	
	Recall that a Hausdorff topological space $X$ is said to be \emph{Tychonoff} if for every closed subspace $F \subset X$ and
	$x \not \in X$, there exists a continuous function $f \colon X \to \R$ such that $f|_F = 0$ and $f(x) = 1$. Note that locally compact Hausdorff spaces are Tychonoff.
	
	\begin{thm}[\cite{abels}*{Theorem 3.3}] \label{thm:abels}
	If $X$ is a Tychonoff $G$-space, then every $x \in X$ admits a trivially proper neighbourhood.
		\qed
	\end{thm}
	
	\begin{lem}[\cite{abels}*{Lemma 3.5}] \label{lem:abels}
		Let $X$ be a locally compact Hausdorff $G$-space and $x \in X$.
		If $(U,f)$ is a trivially proper neighbourhood of $x$, then $S = f^{-1}(\overline 1) \subset U$ is a $G_x$-invariant subspace and $G \times_{G_x} S \to U$, $[g,s] \mapsto g \cdot s$ is a $G$-equivariant homeomorphism.
		\qed
	\end{lem}
	
	\begin{lem}\label{lem:orbi-loceq}
	Let $X$ be a $G$-space. Let $U,U' \subset X$ be two $G$-invariant open neighbourhoods of a point $x \in X$, and let $Z$ be a $G$-orbit. For every pair of continuous equivariant maps $f \colon U \to Z$ and $g \colon U' \to Z$ such that $f(x) = g(x)$, there exists a $G$-invariant neighbourhood $U'' \subset U \cap U'$ of $x$ such that $f|_{U''} = g|_{U''}$.
	\end{lem}
	\begin{proof} Since $Z$ is discrete, the set $U'' = \{u \in U \cap U' : f(u) = g(u)\}$ is an open subspace of $U \cap U'$.
	\end{proof}
	
	\begin{lem}\label{lem:nonemp-cofi} 
	Let $X$ be a locally compact Hausdorff $G$-space, $K \subset X$ a $G$-invariant compact subspace, and $Z$ an orbit. If $f \colon X \to Z$ is an equivariant map, then there exists a $G$-invariant open subspace $U \supset K$ and a map $\widehat f \colon U \to Z$ extending $f$.
	\end{lem}
	\begin{proof} For each $x \in K$, fix a trivially proper neighbourhood $f_x \colon U_x \to G \cdot x$. Since $x \in K$, we may compose with the map $f|_{G \cdot x} \colon G \cdot x \to Z$ to obtain a map $f_x \colon U_x \to Z$ that agrees with $f$ at $x$. Applying Lemma \ref{lem:orbi-loceq} to the restrictions of $f_x$ and $f$ to $U_x \cap K$ if necessary, we may without loss of generality assume that $f_x$ agrees with $f$ in $U_x \cap K$. Since $X$ is locally compact, each $U_x$ contains a compact neighborhood of $x$. Hence we may assume that $U_x$ has compact closure, that function $f_x$ is defined on $\overline{U_x}$, and furthermore that it agrees with $f$ on $\overline U_x \cap K$.
	
	We now extract from the cover above a finite cover $U_{x_1}, \dots, U_{x_n}$ of $K$. Write $U_i = U_{x_i}$ and $f_i = f_{x_i}$. For each $i \ne j$, we put 
	\[
	T_{ij} = \{z \in \overline U_i \cap \overline U_j : f_i(x) \ne f_j(x)\}.
	\]	
	Since $Z$ is discrete, this set is closed in $\overline U_i \cap \overline U_j$ and thus it is closed in $X$. It follows that the subspaces $V_i = U_i \setminus \bigcup_{1 \le j \le n, j \ne i} T_{ij}$ are open. Furthermore, since each pair of functions $f_i$, $f_j$ agree on $\overline U_i \cap \overline U_j \cap K$, we have that $K \subset V_1 \cup \cdots \cup V_n$. 
	It remains to note that, by construction, the functions $(f_i)|_{V_i}$ assemble into a well-defined function $\widehat f \colon \bigcup_{i=1}^n V_i \to Z$ extending $f$.	 
	\end{proof}

	\subsection{Stratification by orbit types}
	
	Recall that two subgroups $H, K \le G$ are \emph{conjugate} if there exists $g \in G$ such that $gHg^{-1} \subset K$. This defines an equivalence relation on the set of subgroups of $G$; we denote the conjugacy class of $H$ by $(H)$. 
	
	The set $P_G$ of conjugacy classes comes equipped with a partial ordering: we say that $(H) \le (K)$ if $H$ is \emph{subconjugate} to $K$, that is, if $H$ is conjugate to a subgroup of $K$. It is straightforward to check that this is a well-defined poset. In particular we may view $P_G$ as a space via the Alexandroff topology.	
	Given a $G$-space $X$ and $H \le G$, we put $X_{(H)} = \{x \in X : (G_x) = (H)\}$ and $X_{\le (H)} = \bigcup_{(K) \le (H)} X_{(K)}$.
	
	\begin{lem} \label{lem:stabi-abels}
	Let $X$ be a $G$-space and $x\in X$. If $f \colon U \to G\cdot x$
	is a trivially proper neighbourhood, then:
		\begin{enumerate}[(i)]
			\item if $y \in f^{-1}(x)$, then $G_y \subset G_x$;
			\item $U \subset X_{\le (G_x)}$;
			\item if we write $U^{G_x}$ for the ${G_x}$-fixed points of $U$, 
			then $X_{(G_x)} \cap f^{-1}(x) \subset U^{G_x}$.
		\end{enumerate}
	\end{lem}
	\begin{proof} 
	Since $f$ is an equivariant map, we know that $G_y \le G_{f(y)}$ for all $y \in U$. This proves (i). Further, since $G \cdot x$ is a transitive $G$-space, all its stabilizers groups are conjugate; in particular for all $y \in U$ we have that $G_{f(y)}$ is subconjugate to $G_{f(x)} = G_x$, proving (ii). 
	Finally we prove (iii). Using (i), the stabilizer of a point $y$ in $X_{(G_x)} \cap f^{-1}(x)$ is both conjugate to and contained in $G_x$. The subgroup $G_x$ is finite; hence a subgroup $H \le G_x$ is conjugate to $G_x$ if and only if $H=G_x$. Therefore $G_y = G_x$, which proves that $y$ is a fixed point for the restricted $G_x$-action on $U$.
	\end{proof}
	
	\begin{lem} \label{lem:orbitproj} For each $G$-space $X$, the map $\pi\colon X \to P_G$, $x \mapsto (G_x)$ is continuous.
	\end{lem}
	\begin{proof}
		It suffices to see that for each $(H) \in C_H$ the set $X_{\le (H)}$ is open. Let $x \in X_{(K)}$ with $(K) \le (H)$ and consider a trivially proper neighbourhood $(U,f)$ of $x$. By Lemma \ref{lem:stabi-abels}, it follows that $x \in U \subset X_{\le (G_x)} = X_{\le (K)} \subset X_{\le(H)}$. Hence $X_{\le (H)}$ is open as claimed.
	\end{proof}
	
	\begin{rmk} Since the map of Lemma \ref{lem:orbitproj} is $G$-equivariant, it also descends to a continuous map $X/G \to P_G$.
	\end{rmk}
	
	\subsection{Covering dimension}
	
	Let $n \in \N$. A topological space $X$ has \emph{covering dimension $\le n$} if every open cover admits a refinement such that every intersection of $n+2$ of its members is empty. 
	This condition will be of relevance to us because it guarantees certain well-behavedness of the $\infty$-topos of sheaves on $X$, namely that it is \emph{hypercomplete} (see \cite{htt}*{Section 7.2.3}):
	
	\begin{thm}[\cite{htt}*{Corollary 7.2.1.12 and Theorem 7.2.3.6}, \cite{som}*{Lemma 3.6.13}] 
	\label{thm:dualstalks}	
	Let $X$ be a locally compact
    Hausdorff space of finite covering dimension and let $C$ be a compactly assembled category (e.g. $C = \An$ or $C$ dualizable).
	A map $\cF\to \cG$ in $\Shv(X,C)$ is an equivalence if and only if the induced maps on stalks $\cF_x \to \cG_x$ are equivalences for all $x \in X$.
	\qed
	\end{thm}
	
	\begin{prop} \label{prop:cov} If $X$ is a paracompact $G$-space of covering dimension $\le n$, then $X/G$ has covering dimension $\le |G|(n+2)-2$. 
	\end{prop}
	\begin{proof} Write $q \colon X \to X/G$ for the canonical quotient map and let $\{U_i\}_{i \in I}$ be an open covering of $X/G$.
		By hypothesis, there exists an open refinement $\{W_j\}_{j \in J}$ of $\{q^{-1}(U_i)\}_{i \in I}$, with refinement function $\alpha \colon J \to I$, and such that for every $F \subset J$ with $|F| \ge n+2$ we have $\bigcap_{f \in F} W_f = \emptyset$. 
		
		Since $q \colon X \to X/G$ is open and surjective, the collection $\{q(W_j)\}_{j \in J}$ is an open cover of $X/G$, and furthermore it is a refinement of $\{U_i\}_{i \in I}$ with refinement function $\alpha$. 
		Consider now a subset $F \subset J$ such that $|F| \ge |G|(n+2)-2+2 = |G|(n+2)$. To conclude the proof we ought to prove that $\bigcap_{f \in F} q(W_f) = \emptyset$. Once again by the surjectivity of $q$, we may prove that
		\[
		\emptyset = q^{-1}\Big(\bigcap_{f \in F} q(W_f)\Big) = \bigcap_{f \in F} q^{-1}q(W_f) = \bigcap_{f \in F} \bigcup_{g \in G} gW_f = \bigcup_{\phi \colon F \to G} \bigcap_{f \in F} \phi(f) W_f.
		\]
		Hence the proof reduces to showing that for each function $\phi \colon F \to G$ with $F \subset J$ and $|F| \ge |G|(n+2)$, we have that $\bigcap_{f \in F}\phi(f) W_f = \emptyset$. For each $g \in G$, this intersection is contained in $g\Big(\bigcap_{f \in \phi^{-1}(g)}W_f\Big)$. It thus suffices for $\phi$ to have a fiber of cardinality greater or equal than $n+2$, which follows from the pigeonhole principle. 
	\end{proof}
	
	Recall that if $J$ is a (right) $G$-set, we have a (left) $G$-compact Hausdorff space $[0,1]^J$ with action $(g \cdot \phi)(x) = \phi(x \cdot g)$. 
	
	\begin{coro} \label{coro:gcubes-hyp}
		If $S$ is a finite right $G$-set, then the orbit space of $[0,1]^S$
		has finite covering dimension. 
	\end{coro}
	\begin{proof}Immediate from the fact that $[0,1]$ has finite covering dimension and Proposition \ref{prop:cov}.
	\end{proof}
	
	We can always equivariantly embed a compact Hausdorff $G$-space into $[0,1]^J$ for some possibly infinite $G$-set $J$.
	
	\begin{lem} \label{lem:g-urysohn} If $X$ is a compact Hausdorff (left) $G$-space, then there exists a right $G$-set $J$ and an equivariant embedding $e \colon X \to [0,1]^J$. 
	\end{lem}
	\begin{proof} Set $J = C(X,[0,1])$ equipped with its usual right action, namely $(f\cdot g)(x) = f(g \cdot x)$. The product $[0,1]^J$ carries a canonical left $G$-space structure via $(g \cdot \phi)(f) = \phi(f \cdot g)$.
		By Urysohn's lemma, we know that $e(x)(f) \defeq f(x)$ is an embedding, and it is equivariant by construction.
	\end{proof}
	
	\begin{rmk}
	Given a $G$-set $J$, the cube $[0,1]^{J}$ is a cofiltered limit of finite cubes $[0,1]^S$ with $S \subset J$.	
	\end{rmk}
	
	\subsection{Cofiltered limits}
	
	Recall that a partially ordered set $I$ is \emph{cofiltered}
	if every finite subset has a lower bound. 
	
	\begin{lem} \label{lem:lims-mod-G}
	The orbits functor $(-)/G \colon \GCH \to \CHaus$ preserves cofiltered limits.
	\end{lem}
	\begin{proof} Consider a cofiltered diagram $(X_i)_{i \in I}$ with transition maps 
		$(\alpha_{ij})_{i \le j}$ and put $X = \lim_{I \ni i} X$.
		The map $(\lim_{I \ni i}X_i)/G \to \lim_{I \ni i}(X_i/G)$ has compact Hausdorff domain and codomain, and hence
		it suffices to verify that it is a bijection.
		Write $q_i \colon X_i \to X_i/G$ and $q \colon X \to X/G$ for the canonical quotient maps and 
		$\pi_i \colon X \to X_i$, $\rho_i \colon \prod_{i \in I} (X_i/G) \to X_i/G$ 
		for the projections. 
		
		We argue for injectivity and surjectivity separately. 
		Since $\lim_{I \ni i}(X_i/G)$ embeds into $\prod_{i \in I} (X_i/G)$, to prove injectivity it suffices
		to see that the map $X/G \to \prod_{i \in I} (X_i/G)$ is injective. In other words, 
		given $x,y \in X$ such that $gx \ne y$ for all $g \in G$, we must prove that there exists $i \in I$ 
		such that $g\pi_i(x) \ne \pi_i(y)$ for all $g \in G$. For a given $g \in G$, the fact that $gx \ne y$ implies that 
		there exists $i_g \in I$ such that $g\pi_{i_g}(x) = \pi_{i_g}(gx) \ne \pi_{i_g}(y)$. 
		Using that $I$ is cofiltered and $G$ is finite, we may choose $i \in I$ such that 
		$i \le i_g$ for all $g \in G$ and hence $g\pi_i(x) \ne \pi_i(y)$ for all $g \in G$.
		
		Now we prove surjectivity. It suffices to check that the composition $X 
		\to X/G \to \lim_{I \ni i} (X_i/G)$ has non-empty fibers. This follows from the fact that the fibers are 
		of the form $\lim_{I \ni i} G \cdot x_i$ and cofiltered diagrams of finite non-empty sets are non-empty.
	\end{proof}
	
	\begin{lem} \label{lem:compcofi}
	Let $X = \lim_{I \ni i} X_i$ be a cofiltered diagram of Hausdorff $G$-spaces. Write $p_i \colon X \to X_i$ and $\alpha_{ij} \colon X_i \to X_j$ for the projection and transition maps respectively. If $K \subset X$ is a compact subspace, then $K = \lim_{I\ni i} p_i(K)$.
	\end{lem}
	\begin{proof} It suffices to see that the canonical map $\varphi \colon K \to \lim_{I \ni i} p_i(K)$ is bijective, since has compact domain and Hausdorff codomain. Injectivity follows from the fact that the projection maps $p_i \colon X \to X_i$ are jointly monomorphic. For surjectivity, let $k = (k_i) \in \lim_{I \ni i} p_i(K)$ and let us see that $\varphi^{-1}(k)$ is non-empty. For each $j \in I$, consider the subspace $A_j = \{z \in \lim_{I \ni i}p_i(K) : z_i = k_i\}$. Since $\{k\} = \cap_{j \in J} A_j$ and $K$ is compact, it suffices to see that $\varphi^{-1}(A_{j_1} \cap \cdots \cap A_{j_s})$ is non-empty for each finite subset $F = \{j_1,\dots, j_s\} \subset J$. Indeed, if we take a lower bound $j_0$ of $F$, since $k_{j_0} \in p_{j_0}(K)$ there exists $w \in K$ such that $p_{j_0}(w) = k_{j_0}$. Now
	\[
	\varphi(w)_{j_l} = p_{j_l}(w) = \alpha_{j_0,j_l}(p_{j_0}(w)) = 
	\alpha_{j_0,j_l}(k_{j_0}) = k_{j_l}
	\]
	for each $l \in \{1,\dots, s\}$ and thus $w \in \varphi^{-1}(A_{j_1} \cap \cdots \cap A_{j_s})$. This concludes the proof. 
	\end{proof}
	
	\begin{lem}\label{lem:nonemp-open}
	Let $X = \lim_{I \ni i} X_i$ be a cofiltered limit of compact Hausdorff $G$-spaces with surjective projection maps $p_i \colon X \to X_i$ and let $Z$ be a $G$-orbit. If $f \colon X \to Z$ is an equivariant map, then there exists $i \in I$ and an equivariant map $f_i \colon X_i \to Z$ such that $f = f_i p_i$.	
	\end{lem}
	\begin{proof} Since the projections $p_i$ are surjective, any factorization of $f$ through some map $p_i$ will automatically be also $G$-equivariant.
		
	Put $F_z = f^{-1}(z)$ for each $z \in Z$. Since $Z$ is discrete, then $F_z$ is a clopen subset of $X$, in particular it is compact. Since $X$ is a cofiltered limit, it has a basis $\mathcal B$ of open sets of the form $p_i^{-1}(U_i)$ for $U_i \subset X_i$, and $\mathcal B$ is closed under finite intersections. Hence for each $z \in X$ there exist $i(z) \in I$ and an open subset $U_z \subset X_{i(z)}$ such that $F_z = p^{-1}_{i(z)}(U_z)$. Taking a lower bound $i$ of $\{i(z) : z \in Z\}$ and replacing $U_z$ by $p^{-1}_{i,i(z)}(U_z)$, we may assume that $i(z)$ is constant. Since $p_i$ is surjective and the collection $(F_z)_{z \in Z} = (p_i^{-1}(U_z))_{z \in Z}$ is disjoint, it follows that $X_i = \bigsqcup_{z \in Z} U_z$. It follows that the constant maps $U_z \to \{z\} \hookrightarrow Z$ assemble into the desired function $f_i \colon X_i \to Z$.
	\end{proof}
	
	\section{Descent conditions}
	
	In this section we define various descent conditions one can impose on a functor $F \colon \GLCH^\op \to C$, and derive some structural consequences of these definitions that will be used throughout the article. 
	
	\begin{defn} A functor $F \colon \GLCH^\op \to C$ is said to satisfy:
	\begin{enumerate}
	\item \emph{open descent} if $F(X) \cong \lim_{I \ni i} F(U_i)$ for each $X \in \GLCH$ and cover $\{U_i\}_{i \in I}$ by $G$-invariant opens of $X$;
	\item \emph{cofiltered compact codescent} if $F$ maps cofiltered limits
	of compact Hausdorff $G$-spaces to colimits;
	\item \emph{closed descent} if $F(\emptyset) = \ast$ and 
	for each $X \in \GLCH$ and 
	$G$-invariant closed subspaces $K,L \subset X$, the square
	\[\begin{tikzcd}[ampersand replacement=\&]
		{F(K \cup L)} \&\& {F(L)} \\
		\\
		{F(K)} \&\& {F(K \cap L)}
		\arrow[from=1-1, to=1-3]
		\arrow[from=1-1, to=3-1]
		\arrow[from=1-3, to=3-3]
		\arrow[from=3-1, to=3-3]
	\end{tikzcd}
	\]
	is a pullback.
	\end{enumerate}
	We will decorate the functor categories $\Fun(\GLCH^\op, C)$ with the subscripts $o$, $cc$ and $cl$ respectively to indicate that we are considering the full subcategory generated by functors satisfying the corresponding descent properties. We shall also consider closed and cofiltered compact codescent for functors $F \colon \GCH^\op \to C$.
	\end{defn}

	By definition, a functor $F \colon \GLCH^\op \to C$ with values in a presentable category $C$ satisfies open descent if and only if it is a sheaf for the Grothendieck topology generated by equivariant open inclusions, which is equivalent to $F_X \defeq F|_{\cO_G(X)^\op}$ being a sheaf for all $X \in \GLCH$. Similarly, if $F \colon \GCH^\op \to C$ satisfies closed descent and cofiltered compact codescent, then the same argument as in
	\cite{som}*{proof of Theorem 3.6.11} shows that $F_X^k \defeq F|_{\cK_G(X)^\op}$ is a $k$-sheaf.
	
	~
	As it turns out, in presence of open and cofiltered compact codescent, closed descent follows formally.
	\begin{lem} \label{lem:o+cc=>cl}
	Let $C$ be a presentable category such that filtered colimits are left exact. If $F \colon \GLCH^\op \to C$ is a functor satisfying open descent
	and cofiltered compact codescent, then for all $X \in \GLCH$ the $k$-sheaf associated to $F_X$ agrees with $F_X^k$. In particular $F$
	satisfies closed descent.
	\end{lem}
	\begin{proof} The $k$-sheaf associated to $F_X$ is given by left Kan extending to $\cKO_G(X)^\op$ and restricting to $\cK_G(X)^\op$. By definition there is a canonical comparison map to $F_X^k$,
	which on objects is given by the formula
	\[
	\colim_{K \subset U} F(U) \to F(K).
	\]
	It remains to observe that the map above is an equivalence. Indeed, by
	cofiltered compact codescent we know that the map $\colim_{K\Subset L} F(L) \to F(K)$ is an equivalence and by a cofinality argument we obtain
	\[
	 \colim_{K \subset U} F(U) \cong \colim_{L \in \cKO_G(X)_{K / }} F(L) \, \cong \colim_{K \Subset L} F(L). \qedhere
	\]
	\end{proof}
	
	\begin{rmk} \label{rmk:stalks} 
	In light of Lemma \ref{lem:o+cc=>cl}, we see that if $C$ is a presentable category such that filtered colimits are left exact and $F \colon \GLCH^\op \to C$ is a functor satisfying open descent, cofiltered compact codescent and closed descent, then the stalk of $F_X$ at $Z \in \Orb$ is given by $F(Z)$.
	\end{rmk}
		
	In a similar spirit to Lemma \ref{lem:o+cc=>cl}, we record the following technical lemma which will be of importance later on.
	
	\begin{lem}\label{lem:ko} 
	Let $C$ be a presentable category and $X$ a locally compact Hausdorff $G$-space. Let $F, F' \colon \mathcal{KO}_G(X)^{\op} \to C$ be two functors. If $\eta \colon F \to F'$ is a natural transformation such that
	$\eta_K$ is an equivalence for all $K \in \cK_G(X)$, then the induced map
	$F|_{\cO_G(X)^\op} \to F'|_{\cO_G(X)^\op}$ is an equivalence upon sheafification.
	\end{lem}
	\begin{proof} 
	By the identifications of Remark \ref{rmk:OG-OX/G}, we may prove
	the statement for $G=1$. We consider the following partial order on $\cKO(X)$: write $A \ll B$ if either $A = B$ or $\overline{A}$ is compact and $\overline{A} \Subset B$.
	Write $\Psh^{\mathrm{cont}}(X,C)$ for the
	subcategory of $\Psh(X,C)$ spanned by presheaves $H$ satisfying 
	$H(U) \cong \lim_{V\ll U} H(V)$. As noted in \cite{efiloc}*{Section 6.2}, there is a canonical reflector $(-)^{\mathrm{cont}}$ to this inclusion, which factorizes the sheafification functor (we point out that in loc. cit. it is assumed that $C$ is dualizable but this is not needed for the result in question). In particular, there is a natural map $F|_{\cO(X)^\op} \to F|_{\cO(X)^\op}^{\mathrm{cont}}$ which is an equivalence upon sheafification. Hence it suffices to see that for each $U \in \cO(X)$ the map $\lim_{V\ll U} F(V) \to \lim_{V\ll U} F'(V)$ induced by $\eta$ is an equivalence. By a finality argument, there are equivalences
	\[
	\lim_{V\ll U} F(V) \cong \lim_{L \in \cKO(X), L \ll U} F(L)
	\cong \lim_{K \subset U} F(K)
	\]
	and likewise for $F'$. Since $\eta$ is an equivalence on compact subspaces, the conclusion follows.
	\end{proof}
	
	Next we turn to comparing functors satisfying descent defined on $\GCH$ and on $\GLCH$.
	
	\begin{thm} \label{thm:res-orb-cons}
	Assume that $C$ is a compactly assembled category. Then 
	the restriction functor $\Fun^{o,cc}(\GLCH^\op,C) \to \Fun(\Orb^\op,C)$ is conservative.
	\end{thm}
	\begin{proof} 
	By Lemma \ref{lem:o+cc=>cl} every functor in $\Fun^{o,cc}(\GLCH^\op,C)$ automatically also satisfies closed descent. 
	Let $\mu \colon F \to F'$ be a natural transformation between functors satisfying open descent, cofiltered compact codescent, and closed descent, such that $\mu_Z$ is an equivalence for all $Z \in \Orb$. To see that $\mu_X$ is an equivalence for all $X \in \GLCH$, we may equivalently see that the induced sheaf maps $F_X \to F_X'$ are equivalences. 
		
	By Lemma \ref{lem:o+cc=>cl}, this reduces to showing that $\mu_X$ is an equivalence for all $X \in \GCH$. Furthermore, using Lemmas \ref{lem:g-urysohn} and \ref{lem:compcofi}, we can always
	write a compact Hausdorff $G$-space as a cofiltered limit of compact $G$-invariant subspaces of spaces of the form $[0,1]^S$. Hence it suffices to prove the statement for spaces of the form $Y = [0,1]^S$ with $S$ a finite $G$-set. Using
	Corollary \ref{coro:gcubes-hyp} and Theorem \ref{thm:dualstalks}, we may see that $F_{Y} \to F'_Y$ is an equivalence on stalks. Finally, Remark \ref{rmk:stalks} tells us that the maps on stalks are given by $\mu_Z$ with $Z \in \Orb$, which are equivalences by hypothesis.
	\end{proof}
	
	\begin{coro}\label{cor_extension} 
	If $C$ is a presentable category such that filtered colimits are left exact, then 
	right Kan extension along the inclusion $i \colon \GCH^\op \hookrightarrow \GLCH^\op$ corestricts to an equivalence
		\[
		\Ran(i) \colon \Fun^{cc,cl}(\GCH^\op,C) \to \Fun^{o,cc}(\GLCH^\op,C)=\Fun^{o,cc,cl}(\GLCH^\op,C).
		\]
	\end{coro}
	\begin{proof} Since $i$ is fully faithful, we know that right Kan extending along $i$ is again a fully faithful functor, and since the inclusion $\Orb^\op \hookrightarrow \GLCH^\op$ factors through $i$, by Theorem \ref{thm:res-orb-cons} we also know that restriction along $i$ is a conservative functor when restricted to $\Fun^{o,cc,cl}(\GCH^\op,C)$.
	Therefore, once we show that $\Ran(i)$ restricts and corestricts appropriately, it will be a formal consequence of the above that it must be an equivalence.
		
	Let $F \in \Fun^{cc,cl}(\GCH^\op,C)$ and write $\widehat F = \Ran(i)(F)$. It suffices to see that for all such $Y$ the restriction of $\widehat F$ to $\cO_G(Y)^{\op}$ is a sheaf.		
	Since $i$ is fully faithful, we have the following
	commuting diagram:
	\[
		\begin{tikzcd}[ampersand replacement=\&]
			{\mathcal O(Y/G)^{\op}} \& {\mathcal O_G(Y)^{\op}} \\
			{{\mathcal {KO}(Y/G)}^{\op}} \& {{\mathcal {KO}_G(Y)}^{\op}} \& {\GLCH^{\op}} \\
			{\mathcal K(Y/G)^{\op}} \& {\mathcal K_G(Y)^{\op}} \& {\GCH^{\op}} \&\& \Sp
			\arrow["\sim", from=1-1, to=1-2]
			\arrow[hook', from=1-1, to=2-1]
			\arrow[hook', from=1-2, to=2-2]
			\arrow["\sim", from=2-1, to=2-2]
			\arrow["\jmath", hook, from=2-2, to=2-3]
			\arrow["{{\widehat F}}", bend left = 30, from=2-3, to=3-5]
			\arrow[hook, from=3-1, to=2-1]
			\arrow["\sim", from=3-1, to=3-2]
			\arrow[hook, from=3-2, to=2-2]
			\arrow["{{\jmath'}}"', hook, from=3-2, to=3-3]
			\arrow["i", hook, from=3-3, to=2-3]
			\arrow["F", from=3-3, to=3-5]
		\end{tikzcd}
		\]
	Using \cite{htt}*{Theorem 7.3.4.9}, to prove that $\widehat{F}|_{\cO_G(Y)^{\op}}$ is a sheaf it suffices to see that $\widehat F \jmath$ is a right Kan extension of $F \jmath'$. This amounts to proving 
	that for each $K \in \mathcal{KO}_G(Y)$, the inclusion $\mathcal {K}_G(Y)_{/K} \to (\GCH)_{/K}$ is cofinal, which is in turn equivalent to proving that, for equivariant map $f \colon B \to K$ with $B \in \GCH$, the category $(\mathcal{K}_G(Y)_{/K})_{f/} \defeq {\mathcal K}_G(Y)_{/K} \times_{(\GCH)_{/K}} ((\GCH)_{/K})_{f/}$ is weakly contractible. Its objects are equivariant maps $g \colon B \to L$ that factor $f$ through some $G$-invariant compact subspace $L \subset K$, and a (unique) morphism $g \to g'$ exists if and only if the codomain of $g$ is contained in that of $g'$.
	Hence $(\mathcal{K}_G(Y)_{/K})_{f/}$ can be identified with the poset $\{L \in \mathcal K_G(Y) : f(B) \subseteq L \subseteq K\}$, with the order given by inclusion. The latter is contractible since it has a minimum, namely $f(B)$.
	\end{proof}
	
	Before moving on, we record the following result due to Hoyois which says that functors satisfying cofiltered compact codescent 
	are homotopy invariant.
	
	\begin{prop}[Homotopy invariance] \label{prop:htpyinv}
	Let $C$ be a compactly assembled category. 
	\begin{enumerate}[(i)]
		\item If $F \colon \GCH^{\op} \to C$ satisfies cofiltered compact codescent, then it is $G$-homotopy invariant.
		\item If $F \colon \GLCH^{\op} \to C$ satisfies open descent and 
		cofiltered compact codescent, then it is $G$-homotopy invariant.
	\end{enumerate}
	\end{prop}
	\begin{proof} 
	We first prove (i). Write $I = [0,1]$ for the unit interval with trivial action. 
	Since $C$ is compactly assembled, by \cite{som}*{Theorem 2.2.15 (3) and Lemma 2.3.15} there is a fully faithful functor 
	from $C$ into $\Ind(C^{\omega_1})$, namely, the left adjoint of the colimit-realisation functor. The latter category is compactly generated and so in particular there is a jointly conservative family of functors $(F_\alpha)_{\alpha \in \Lambda} \colon C \to \An$ that preserve filtered colimits. Considering the composition of $F$ with each functor $F_\alpha$ and taking homotopy groups, we may without loss of generality assume $C = \Set$.
	
	Fixing $X \in \GCH$ and considering $F' = F(X \times -) \colon \CHaus^{\op} \to \Set$, we
	further reduce the statement to proving that a functor on compactly Hausdorff spaces satisfying compact cofiltered codescent maps the projection $[0,1] \to \ast$ to an equivalence or equivalently maps the functions $i_0 \colon \ast \hookrightarrow I$ and $i_1 \colon\ast \hookrightarrow I$ selecting $0$ and $1$ respectively to the same map. 
	
	For each $s \in I$, write $i_s \colon \ast \hookrightarrow I$ for the map selecting $s$ and consider the function
	\[
	\theta \colon F(I) \to \hom_{\Set}(I,F(\ast)), \qquad \theta_\xi(s) = F(i_s)(\xi).
	\]
	Fix $\xi \in F(I)$. We shall now see that the function $\theta_\xi \colon I \to F(\ast)$ is locally constant and hence constant. For a given $s \in I$ we consider the value $\theta_\xi(s) \in F(\ast)$ and the `constant' object $\xi' \in F(I)$ obtain by pullback along $I \to \ast$. The corresponding map $\theta_{\xi'}$ is the constant function with value $\theta_\xi(s) \in F(\ast)$. We want to argue that $\theta_\xi$ locally around $s$ agrees with $\theta_{\xi'}$.  
	By cofiltered compact codescent, 
	the equality $\cap_{[a,b] \ni s} [a,b] = \{s\}$ induces 
	a bijection
	\[
	\phi_s \colon \colim_{s \in [a,b] \subset I} F([a,b]) \xto{\sim} F(\{s\}).
	\]
	Since $\xi$ and $\xi'$ agree after pullback to $F(\{s\})$, hence in the colimit, they have to agree in a finite stage already, so in some $F([a,b])$. Thus the resulting functions $\theta_\xi$ and $\theta_{\xi'}$ also agree there. This finishes the proof of (i).
	
	To conclude we prove (ii). Note that the projections $X \times I \to X$ for each $X \in \GLCH$ assemble into a natural map $\mu_X \colon F(X) \to F(X \times I)$ between $F$ and $F(- \times I)$, both of which satisfy open descent and compact cofiltered codescent. By Theorem \ref{thm:res-orb-cons}, it suffices to prove that $\mu_Z$ is an equivalence for each $Z \in \Orb$. The latter now follows from applying (i) to $F|_{\GCH^\op}$.
	\end{proof}
	
	\begin{coro}\label{coro:cl+}
	Let $C$ be a compactly assembled category. If $F \colon \GCH^{\op} \to C$ satisfies cofiltered compact codescent and closed descent, then for each map $p \colon X \to Y$ and closed subspace $L \subset Y$, the functor $F$ sends the square
	\[
	\begin{tikzcd}[ampersand replacement=\&]
		{p^{-1}(L)} \& X \\
		L \& Y
		\arrow[from=1-1, to=1-2]
		\arrow["{p|}"', from=1-1, to=2-1]
		\arrow["p", from=1-2, to=2-2]
		\arrow[hook, from=2-1, to=2-2]
	\end{tikzcd}
	\]
	to a pullback square.
	\end{coro}
	\begin{proof} Consider, using Lemma \ref{lem:g-urysohn}, an embedding $e \colon X \to [0,1]^S$ for some $G$-set $S$. By Proposition \ref{prop:htpyinv} and the fact that $[0,1]^S$ is $G$-equivariantly contractible, it suffices to show that $F$ sends
	\[
	\begin{tikzcd}[ampersand replacement=\&]
		{p^{-1}(L)} \& X \arrow[hook]{d}{(p,e)}\\
		L \times [0,1]^S \& Y \times [0,1]^S
		\arrow[hook, from=1-1, to=1-2]
		\arrow[hook, from=1-1, to=2-1]
		\arrow[hook, from=2-1, to=2-2]
	\end{tikzcd}
	\]
	to a pullback square, which follows directly from closed descent.
	\end{proof}
	
	Now we turn to functors satisfying descent on $\GLCH^{pdp}$, the category whose objects are locally compact Hausdorff $G$-spaces and whose maps $X \to Y$ are spans
	\begin{equation}\label{eq:span}
	\begin{tikzcd}
	& U \arrow[hook]{ld}[above]{i\, }\arrow{rd}{p} & \\
	X \arrow{rr}{(i,p)}& & Y 
	\end{tikzcd}
	\end{equation}
	where $U \hookrightarrow X$ is an open $G$-equivariant inclusion and $U \xto{p} X$ is a proper $G$-equivariant map. Composition is given by pullback of spans.

	\begin{defn} Let $C$ be a stable presentable category. We say that a functor $F \colon (\GLCH^{pdp})^{\op} \to C$ 
	satisfies:
	\begin{enumerate}
		\item \emph{cofiltered compact codescent} if its restriction to $\GCH^\op$ satisfies cofiltered compact codescent;
		\item \emph{closed descent} if its restriction to $\GCH^\op$ satisfies closed descent;
		
		\item \emph{open-closed excision} if for each $X \in \GLCH$ and 
		invariant open $U \subset X$, the sequence
		\[
		X \setminus U \xto{(1,j)} X \xto{(i,1)} U
		\]
		in $\GLCH^{pdp}$
		induced by the inclusions $i \colon U \hookrightarrow X$ and $j \colon X \setminus U \hookrightarrow X$ is sent to a fibre sequence by $F$.
		\item \emph{open codescent} if the restriction to $\mathcal{O}(X)^\op$ for each $X$ satisfies codescent. 
	\end{enumerate}
	\end{defn}
	
	\begin{lem} \label{lem:oc=cl}
	Let $C$ be a stable presentable category. If a functor $F \colon (\GLCH^{pdp})^{\op} \to C$ satisfies
	open-closed excision, then it satisfies closed descent.
	\end{lem}
	\begin{proof} 
	Given $X \in \GLCH$ and closed subspaces $K,L \subset X$, we may consider the following diagram in $\GLCH^{pdp}$:
	\[\begin{tikzcd}[ampersand replacement=\&]
		{K \cap L } \& K \& {K \setminus L} \\
		L \& {K \cup L} \& {K \setminus L}
		\arrow[from=1-1, to=1-2]
		\arrow[from=1-1, to=2-1]
		\arrow[from=1-2, to=1-3]
		\arrow[from=1-2, to=2-2]
		\arrow[equals, from=1-3, to=2-3]
		\arrow[from=2-1, to=2-2]
		\arrow[from=2-2, to=2-3]
	\end{tikzcd}\]
	Since $F$ maps both rows to fibre sequences, and the rightmost map is sent to an isomorphism, it follows that $F$ maps the leftmost square to a pullback square. Applying open-closed excision to $\emptyset \to \ast \xto{id} \ast$ we also see that $F(\emptyset) \simeq 0$. 
	\end{proof}
	
	\begin{lem} \label{lem:codesc}
	Let $C$ be a stable presentable category. If a functor $F \colon (\GLCH^{pdp})^{\op} \to C$ satisfies
	open-closed excision and cofiltered compact codescent, then it satisfies open codescent.
	\end{lem}
	\begin{proof}
	We need to show finite open codescent for pushouts and filtered colimits. Finite open codescent for pushouts and for the empty set works exactly as in the last lemma with the roles of open and closed interchanged.  Let $U = \colim_{i \in I} U_i$ a filtered colimit of open subsets of $X$. Replacing $X$ with its compactification we may assume that $X$ is compact. We denote the complement of $U$ in $X$ by $Z$ and similarly the complements of the $U_i$ by $Z_i$.  Then the map $\colim_{i \in I} F(U_i) \to F(U)$ is an equivalence by open-closed excision precisely if the map $\colim_{i \in I} F(Z_i) \to F(Z)$ is an equivalence. The latter follows by cofiltered compact codescent, since $Z = \bigcap_{i \in I} Z_i = \lim_{i \in I} Z_i$. 
	\end{proof}
	There is an equivalence of categories
	\[
	\kappa \colon \GLCH^{pdp} \xto{\sim} (\GCH)_\ast
	\]
	which is given by the one-point compactification on objects, and sends a span \eqref{eq:span} to the map $X^\infty \to Y^\infty$ which agrees with $p$ on $U$ and is constantly $\infty_Y$ on $\{\infty_X\} \cup (X \setminus U)$. In particular, the inclusion $\iota \colon \GCH \hookrightarrow \GLCH^{pdp}$ corresponds to the functor $- \cup \{+\} \colon \GCH\to (\GCH)_\ast$ that freely adjoins a basepoint, and thus it admits a left adjoint $\ell \colon \GLCH^{pdp} \to \GCH$ which is given by $\kappa$ followed by the forgetful functor
	$(\GCH)_\ast \to \GCH$. In view of Lemma \ref{lem:oc=cl},
	for any stable presentable category $D$ the restriction along $\iota$ yields a well-defined functor 
	\[
	\iota^\ast \colon \Fun^{oc,cc}((\GLCH^{pdp})^\op, D) \to \Fun^{cc,cl}(\GCH^\op, D).
	\]
	To conclude this section, we show that this functor is an equivalence whenever $D$ is dualizable.
	
	\begin{thm} \label{thm:pdp}
	If $D$ is a dualizable category, then
	restriction along the the inclusion $\iota \colon \GCH \hookrightarrow \GLCH^{pdp}$ induces an equivalence 
	\[
	\iota^\ast \colon \Fun^{oc,cc}((\GLCH^{pdp})^\op, D) \xto{\sim} \Fun^{cc,cl}(\GCH^\op, D)
	\]
	\end{thm}
	\begin{proof} We shall construct an inverse for $\iota^\ast$. The embedding 
	\[
	\GLCH^{pdp} \xto{\kappa} (\GCH)_\ast \subset \Fun(\Delta^1, \GCH) 
	\cong \Fun((\Delta^1)^\op, \GCH)
	\]
	yields a functor $\hat\kappa \colon \GLCH^{plp} \times (\Delta^1)^\op \to \GCH$. Precomposition by $(\hat \kappa)^\op$ followed by the exponential law and taking fibres allows us to define a functor
	\begin{align*}
	\wp \colon \Fun^{cc,cl}(\GCH^\op, D) 
	\to \Fun((\GLCH^{pdp})^\op, D),
	\end{align*}
	whose formula on objects explicitly reads
	\[
	\wp(E)(X) = \fib(E(X^\infty) \to E(\infty)).
	\]
	Next we shall see that $\wp$ corestricts appropriately, that is, that if $E \colon \GCH^\op \to D$ satisfies cofiltered compact codescent and closed descent, then $\wp(E)$ satisfies cofiltered compact codescent and open-closed excision.
	
	The first condition is immediate from the formula above, the fact that one-point compactification restricted to $\CHaus$ corresponds to $- \cup \{+\}$, and that fibres in a stable category commute with all limits and colimits. We thus turn to open-closed excision. Let $X \in \GLCH^{plp}$, let $U \subset X$ be an invariant open subspace and put $F = X \setminus U$. Applying Lemma \ref{coro:cl+} to $E$, the map $X^\infty \to U^\infty$ and $\infty \hookrightarrow U^\infty$, we obtain that 
	\[
	\begin{tikzcd}[ampersand replacement=\&]
		{E(U^\infty)} \& {E(X^\infty)} \\
		{E(\infty)} \& {E(F^\infty)}
		\arrow[from=1-1, to=1-2]
		\arrow[from=1-1, to=2-1]
		\arrow[from=1-2, to=2-2]
		\arrow[from=2-1, to=2-2]
	\end{tikzcd}
	\]
	is a pullback square, from which it follows that $\wp(E)(U) \to E(X^\infty) \to E(F^\infty)$ is a fibre sequence. Now consider
	the following diagram whose rows are fibre sequences:
	\[\begin{tikzcd}[ampersand replacement=\&]
		{\wp(E)(X)} \& {E(X^\infty)} \& {E(\infty)} \\
		{\wp(E)(F)} \& {E(F^\infty)} \& {E(\infty)}
		\arrow[from=1-1, to=1-2]
		\arrow[from=1-1, to=2-1]
		\arrow[from=1-2, to=1-3]
		\arrow[from=1-2, to=2-2]
		\arrow[equals, from=1-3, to=2-3]
		\arrow[from=2-1, to=2-2]
		\arrow[from=2-2, to=2-3]
	\end{tikzcd}
	\]
	Taking fibres vertically we obtain an equivalence $\fib(\wp(E)(X) \to\wp(E)(F)) \cong \wp(E)(U)$, and a straightforward diagram chase shows that the inverse is induced by the map $\wp(E)(U) \to \wp(E)(X)$. This concludes the proof that $\wp(E)$ satisfies open-closed excision.
	From now on we shall abuse notation and write $\wp$ for its corestriction to $\Fun^{oc,cc}((\GLCH^{pdp})^\op, D)$. 
	
	Now we concentrate on proving that $\wp$ and 
	$\iota^\ast$ are mutual inverses. Given $X \in \GCH^\op$, 
	the inclusion $X \subset X^\infty = X \sqcup \infty$ provides a map
	\[
	i^\ast(\wp(E))(X) = \fib(E(X \sqcup \infty) \to E(\infty)) \to E(X \sqcup \infty) \to E(X)
	\]
	both natural in $E \in \Fun^{cc,cl}(\GCH^\op, D)$ and $X$. This map
	fits as the top-row composition in the following diagram
	\[\begin{tikzcd}[ampersand replacement=\&]
		{(\iota^*\wp)(E)(X)} \& {E(X \sqcup \infty)} \& {E(X)} \\
		0 \& {E(\infty)} \& 0
		\arrow[from=1-1, to=1-2]
		\arrow[from=1-1, to=2-1]
		\arrow[from=1-2, to=1-3]
		\arrow[from=1-2, to=2-2]
		\arrow[from=1-3, to=2-3]
		\arrow[from=2-1, to=2-2]
		\arrow[from=2-2, to=2-3]
	\end{tikzcd}\]
	Since the left hand square is a pullback by definition, and the rightmost square is a pullback because $E$ satisfies closed descent, it follows that the exterior square is a pullback and thus the map $\iota^\ast(\wp(E))(X) \to E(X)$ must be an equivalence.
	
	Finally we consider $\wp(\iota^\ast(F))$ for a given $F \in \Fun^{oc,cc}((\GLCH^{pdp})^\op, D)$. Note that for each $Y \in \GLCH^{plp}$ the composition $\infty \hookrightarrow Y^\infty \to Y$
	yields the zero map, represented by the span $\infty \leftarrow \emptyset \to Y$. Consequently, 
	the inclusion $Y \subset Y^\infty$ defines a map
	\[
	F(Y) \to \wp(\iota^\ast(F))(Y) = \fib(F(Y^\infty) \to F(\infty)) 
	\] 
	natural in both $F$ and $Y$, which is an equivalence by open-closed excision.
	\end{proof}
	
	\section{Definition of Bredon sheaf cohomology}
	
	Let $C$ be a presentable category. Consider the category $\GTop$  of topological spaces with a $G$-action as a Grothendieck site equipped with the topology generated by jointly surjective equivariant open inclusions, and $\Orb$ equipped with the indiscrete Grothendieck topology. 
	The inclusion $t \colon \Orb \hookrightarrow \GTop$ is a functor of sites and thus defines a morphism at the level of sheaves, which has a left adjoint given by left Kan extending and sheafifying:
	\begin{equation}\label{def:jstar}
		t^\ast = \sh \circ \Lan(i) \colon \Psh(\Orb,C)  \longleftrightarrow \Shv(\GTop,C) \colon t_\ast =: (-)|_{\Orb}.
	\end{equation}
	Note that it might be slightly counterintuitive to denote this functor by $t^\ast$ since it is the same direction as $t$, but this is the convention for morphisms of sites, which already are in some sense in the opposite direction to geometric morphisms.
	
	We also note that the site $\GTop$ is large, so that a priori one might run into size issues here. However, since $
\Orb$ is small and the sheafification only involves open subsets of a given $X \in \GTop$, which is a small category, this is not an issue (we will see this concretely from the formula after the next definition). 
	
	\begin{defn}\label{def_bredon}
	Given a functor $E \in \Psh(\Orb,C)$ and a $G$-space $X$, we define \emph{Bredon cohomology with coefficients in $E$} as $\bre(X,E) \defeq t^\ast(E)(X)$.
	\end{defn}

	Throughout this section we fix $E \in \Psh(\Orb)$. For a given $X \in \GTop$ we put $\preuE_X$ for the presheaf given by the restriction of $\Lan(t)(E)$ to $\cO_G(X)^{\op} \simeq \cO(X/G)^{\op}$, and $\uE_X$ for its sheafification, which agrees with the restriction of $\bre(-,E)$. Concretely $\preuE_X \in 
\Shv(X/G, C)$ is given by 
	\[
U \longmapsto \colim_{\,q^{-1}(U)\to Z} E(Z),
\]
where the colimit ranges over all $G$--equivariant maps from
$q^{-1}(U)$ to orbits $Z \in \Orb$.

       \begin{rmk}
       In what follows, we will primarily consider $G$-spaces $X$ that are locally compact Hausdorff, and the reader may safely restrict attention to this case. Nevertheless, the definition applies to arbitrary $G$-spaces, and we will need to allow non–locally compact spaces when comparing with singular Bredon cohomology in Section~\ref{subsec:bredoncw}.
        \end{rmk}
	
	\begin{ex}\label{ex:bredon-trivial}
	If $G$ acts trivially on $X$, then $\bre(X,E)$ is given by sheaf cohomology of $X$ with value in $E(G/G)$:
	\[
	\bre(X,E) = \Gamma(X, E(G/G))
	\]
	This follows since in this case  for every non-empty open $U$ the category of all maps $U \to Z$ is trivial, i.e. equivalent to a singleton given by $Z = G/G$ and the unique map $U \to Z$ , since $U$ cannot map to
	an orbit with non-full isotropy. 
    \end{ex}

	The counit of the adjunction \eqref{def:jstar} provides us with a natural map
	\begin{equation}\label{bredon-counit}
		\bre(-,\mathcal H|_{\Orb}) \to \mathcal H,
	\end{equation}
	for any sheaf $\mathcal{H} \in \Shv(\GLCH,C)$
	whereas the unit of the adjunction yields a map
	\begin{equation}\label{bredon-unit}
	E \to \bre(-,E)|_{\Orb}.
	\end{equation}
	
	One could ask how much information is lost by restricting $\mathcal H$ and then considering the associated Bredon homology. 
	Before addressing this question, we note that the value at orbits is not modified.
	
	\begin{lem} \label{lem:shorb}
	For each $\cF\in \Psh(\GTop,C)$, the map $\cF\to \cF^{\sh}$ is an equivalence for all $Z \in \Orb$, where $ \cF^{\sh}$ denotes the sheafification.
	\end{lem}
	\begin{proof} We may reduce to prove the statement for $\An$-valued sheaves, from which the desired conclusion is obtained by tensoring by $C$. Now we may apply \cite{pstr}*{Proposition 4 of Appendix A} to see that restriction to orbits commutes with sheafification, concluding the proof.
	\end{proof}
	
	\begin{lem} \label{lem:bredon-counit} The map \eqref{bredon-counit} is an equivalence for all $Z \in \Orb$.
	\end{lem}
	\begin{proof} 
	Since $t$ is fully faithful, so is $\Lan(t) \colon \Psh(\Orb,C) \to \Psh(\GTop,C)$, and hence the corresponding unit map is an equivalence. It follows from that and from the triangle identities that $\Lan(t)(\mathcal H|_{\Orb})|_{\Orb} \to \mathcal H|_{\Orb}$ is an equivalence. 
	Since the map \eqref{bredon-counit} is given by the composition $\sh(\Lan(t)(\mathcal H|_{\Orb})) \to \sh(\mathcal H) \xto{\sim} \mathcal H$, it suffices to prove that a map $\cF \to \cF'$ of presheaves on $\GTop$ which is an equivalence on orbits remains so upon sheafification, which is an immediate consequence of Lemma \ref{lem:shorb}. 
	\end{proof}
	
	\begin{lem} \label{lem:bredon-unit}
	The map \eqref{bredon-unit} is an equivalence.
	\end{lem}
	\begin{proof} The unit map is given by the composition
	\[
	E \to \Lan(t)(E)|_{\Orb} \to \Lan(t)(E)^\sh|_{\Orb}.
	\]
	The first map is an equivalence by the fully faithfulness of $\Lan(t)$ and the second one by Lemma \ref{lem:shorb}.	
	\end{proof}
	
	We now want to prove a generalization of the previous statement, that computes Bredon sheaf cohomology for orbits. 
	
	 \begin{prop}\label{lem:ind-res}
    Let $H \subseteq G$ be a subgroup and let $X$ be an $H$-space. Then we can consider the induced $G$-space
    $G \times_H X$ and get
    \[
    \bre(G \times_H X,E) \simeq \Gamma^H_{\mathrm{Br}}(X, E|_H)
    \]
    where $E|_H$ is the restriction of $E$ along the induction functor $\mathrm{Orb}_H \to \Orb$. 
    \end{prop}
    \begin{proof}
    We first observe that 
    $\Psh(\Orb,C) = \Fun^{\times}(\mathrm{Fin}_G^\op, C)$ where $\mathrm{Fin}_G$ is the category of finite $G$-sets and $\Fun^{\times}$ denotes finite product preserving functors. Now we have an adjunction 
    \[
    \Fun^{\times}(\mathrm{Fin}_G^\op, C) \leftrightarrow \Fun^{\times}(\mathrm{Fin}_H^\op, C)
    \]
    where the left adjoint is restriction along the functor 
    \[
    \mathrm{Fin}_H \to \mathrm{Fin}_G \qquad S \mapsto G \times_H S
    \]    
    and the right adjoint is restriction along the forgetful functor $\mathrm{Fin}_G \to \mathrm{Fin}_H$. This follows directly from the fact that these two functors on the indexing categories are adjoint to one another and both preserve coproducts. We have a similar adjunction on the left of $G$ and $H$-spaces:
     \[
    \Shv(\mathrm{Top}_G, C) \leftrightarrow \Shv(\mathrm{Top}_H, C) \ .
    \]
    Again induced by an adjunction on the level of indexing categories. 
    Now the assertion is that the diagram of left adjoint functors
    \[
\begin{tikzcd}
    \Fun^{\times}(\mathrm{Fin}_G^\op, C) \arrow[r, "t^*"] \arrow[d, "(-)|_H"'] & \Shv(\mathrm{Top}_G, C) \arrow[d, "h"] \\
    \Fun^{\times}(\mathrm{Fin}_H^\op, C)  \arrow[r, "t^*"'] &  \Shv(\mathrm{Top}_H, C)
\end{tikzcd}
\]
commutes. This is equivalent to the commutation of the right adjoints, which is obvious since it comes down to the commutativity of the diagram
    \[
\begin{tikzcd}
    \mathrm{Fin}_G \arrow[r, "t"] \arrow[d] & \mathrm{Top}_G \arrow[d] \\
    \mathrm{Fin}_H  \arrow[r, "t"'] &  \mathrm{Top}_H
\end{tikzcd}
\]
 on the level of indexing categories.
\end{proof}

	\begin{rmk}\label{rem_genuine}
	We can in fact refine Bredon sheaf cohomology $\bre(X, E)$ to an object $\breg(X,E) \in \Psh(\Orb)$ whose value at $G/G $ is given by
	$\bre(X, E)$ and whose value on $G/H$ is given by $\bre(X \times G/H, E)$. In other words:  a (naively) genuine $G$-spectrum whose genuine $G$-fixed points are $\bre(X, E)$ and whose underlying spectrum is $\bre(X \times G, E) = \Gamma(X, E(G/e))$, i.e. sheaf cohomology of $X$ with value in the underlying spectrum $E(G/e)$.
	\end{rmk}	
	
	\section{Properties of Bredon sheaf cohomology}
	
	Below we shall describe the stalks of the sheaves $\uE_X$,
	prove that Bredon sheaf cohomology satisfies cofiltered compact codescent
	and derive several structural consequences. To this end, 
	we crucially rely on the following lemma.
	
	\begin{lem}\label{lem:technical} 
	Let $C$ be a presentable category and $E \in \Psh(\Orb,C)$.
	\begin{enumerate}[(i)]
		\item If $X = \lim_{I \ni i} X_i$ is a cofiltered limit 
		in $\GLCH$ with projection maps $p_i \colon X \to X_i$,
		then for all compact subspaces $K \subset X$ there is an equivalence
		\begin{equation}\label{eq:profineq}
			\colim_{I^\op \ni i} \colim_{p_i(K) \subset V} \colim_{V \to Z} E(Z) \xto{\sim} \colim_{K \to Z'} E(Z')
		\end{equation}
		where the colimits are ranging over the poset $I$, equivariant open subspaces of $X_i$ containing a given $p_i(K)$, $((\Orb)_{V/})^\op$ and $((\Orb)_{K/})^\op$ respectively.
		\item if $X$ is a Tychonoff $G$-space, then for all orbits $Z \subset X$
		there is an equivalence
		\begin{equation}\label{eq:stalkco}
			\colim_{U\supset Z} \colim_{U \to Z'} E(Z') \xto{\sim} E(Z)
		\end{equation}
		where the colimits range over equivariant open subspaces of $X$ containing $Z$ and $((\Orb)_{U/})^\op$.
	\end{enumerate}
	\end{lem}
	\begin{proof} We first prove (i). Write $K_i = p_i(K)$ for all $i \in I$.
	We wish to study the map \eqref{eq:profineq} by means of a cofinality argument, for which we introduce a (1-)category $A_K$ as follows. Its objects are given by tuples $(i,U,f)$ where $i \in I$, $K_i \subset U \subset X_i$ is a $G$-invariant open and $f \colon U \to Z$ an equivariant map to an orbit. The set of morphisms $(i,U,f) \to (j,V,g)$ can be non-empty only if $i < j$ and $U \subset p_{ij}^{-1}(V)$, in which case it is given by equivariant maps $\alpha \colon Z \to Z'$ that make 
	the square
	\[
	\begin{tikzcd}
		U \arrow{d}[left]{p_{ij}} \arrow{r}{f} & Z \arrow{d}{\alpha} \\
		V \arrow{r}{g} & Z' 
	\end{tikzcd}
	\]
	commute. 
	
	There is a canonical comparison functor
	\[
	F \colon A_K \to (\Orb)_{K /}, \qquad (i,U,f) \mapsto (K \xto{p_i} K_i \hookrightarrow U \xto{f} Z).
	\]
	If we write $\pi \colon (\Orb)_{K /} \to \Orb$ for the canonical projection and $\pi' = \pi \circ F$, then the triangle
	\begin{equation} \label{triang:overorb}
		\begin{tikzcd}
			A_K \arrow{rr}{F} \arrow{dr}[below]{\pi'\quad} & &  \arrow{dl}{\pi} (\Orb)_{K/} \\
			& {\Orb}
		\end{tikzcd}
	\end{equation}	
	commutes by definition and we can 
	describe \eqref{eq:profineq} as the map induced by precomposition by $F^\op$:
	\[
	\colim_{A_K^\op} E \circ \pi^\op \circ F^\op \to \colim_{(\Orb)_{K /}^\op} E \circ \pi^\op.
	\]	
	It thus suffices to show that $F^\op$ is colimit-final, or equivalently, that $F$ is limit-final. Moreover, we shall 
	see that $F$ is a Dwyer-Kan localization.
	
	In light of \eqref{triang:overorb} and \cite{kerodon}*{Proposition 6.3.4.2 02LW}, we shall check that both $\pi$ and $\pi'$ are cocartesian fibrations, and that $F$ preserves cocartesian edges; this will reduce the proof to showing that for each $Z \in \Orb$ the induced functor on fibers
	\begin{equation}\label{eq:F-fiberwise}
		F_Z \colon (A_K)_{Z} \to ((\Orb)_{K/})_Z
	\end{equation}
	is a Dwyer-Kan localization. 
	
	It is straightforward to verify that all morphisms in $(\Orb)_{K/}$ are $\pi$-cocartesian and that $\pi$ is a cocartesian fibration; in particular $F$ preserves
	cocartesian edges. To prove that $\pi'$ is a cocartesian fibration we observe that for each $(i,U,f \colon U \to Z) \in A_K$ we may lift any $\alpha \colon Z \to Z'$ in $\Orb$ to a map $\alpha \colon (i,U,f) \to (i,U,\alpha f)$, which is $\pi'$-cocartesian. 
	
	We now turn to showing \eqref{eq:F-fiberwise} for a given $Z \in \Orb$. Since the fiber $((\Orb)_{K/})_Z$ is a discrete anima indexed by equivariant maps $f \colon K \to Z$, it suffices 
	to see that each fiber $((A_K)_Z)_{/f}$ is weakly contractible.
	We shall see that  $((A_K)_Z)_{/f}$  is a cofiltered category, which in particular implies that it is contractible.
	
	Let us first spell out what the objects and morphisms of $((A_K)_Z)_{/f}$ are.
	An object in $((A_K)_Z)_{/f}$ is given by an object $(i,U,g \colon U \to Z) \in A_K$ fitting in the following triangle
	\[
	\begin{tikzcd}[ampersand replacement=\&]
		U \& Z \\
		K
		\arrow["g", from=1-1, to=1-2]
		\arrow["{p_i|}", from=2-1, to=1-1]
		\arrow["f"', from=2-1, to=1-2]
	\end{tikzcd}
	\]
	There is at most one arrow $(i,U,g) \to (j,V,h)$ between two objects, which exists whenever $h p_{ij}| = g$. In particular, in showing that $((A_K)_Z)_{/f}$ is cofiltered we need not consider parallel pairs of arrows.
	
	We first show that $((A_K)_Z)_{/f}$ is non-empty. By Lemma \ref{lem:compcofi} we know that $K = \lim_{I \ni i} K_i$, and applying Lemma \ref{lem:nonemp-open} to this limit we obtain
	that there exists some $i \in I$ for which $f$ factors through $p_i| \colon K \to K_i$ and an equivariant map $g' \colon K_i \to Z$. Now Lemma \ref{lem:nonemp-cofi} applied to $X_i \supset K_i$ and $g'$ guarantees that there exists an equivariant open $U \supset K_i$ and an extension $g \colon U \to Z$ of $g'$. By construction $(i,U,g)$ defines an object in $((A_K)_Z)_{/f}$. 
	
	At last, given two objects $(i,g \colon U \to Z)$, $(j,V,g'\colon V \to Z)$, consider the following solid arrow diagram
	over $f \colon K \to Z$:
	\[
	\begin{tikzcd}
		&& {K_i} & U \\
		K & {K_s} & W &&& Z \\
		&& {K_j} & V
		\arrow[hook, from=1-3, to=1-4]
		\arrow["g", from=1-4, to=2-6]
		\arrow["{p_i}", bend left=10, from=2-1, to=1-3]
		\arrow["{p_s}", dashed, from=2-1, to=2-2]
		\arrow["{p_j}"', bend right =10, from=2-1, to=3-3]
		\arrow[dashed, hook, from=2-2, to=2-3]
		\arrow["{p_{si}}|", dashed, from=2-3, to=1-4]
		\arrow[dashed, "g''", bend left = 10, from=2-3, to=2-6]
		\arrow["{p_{sj}}|"', dashed, from=2-3, to=3-4]
		\arrow[hook, from=3-3, to=3-4]
		\arrow["g'"', from=3-4, to=2-6]
	\end{tikzcd}
	\]
	We construct the dotted arrows as follows. Take $s \in I$ below $i$ and $j$ and put $W' = p_{si}^{-1}(U) \cap p_{sj}^{-1}(V)$. 
	Since $g p_{si}| \colon W' \to Z$ and $g p_{sj}| \colon W \to Z$ are continuous and $Z$ is discrete, the equivariant subspace $W = \{w \in W' : g p_{si}(w) = g' p_{sj}(w)\} \subset W'$ contains $K_s$ and is open in $W'$; in particular it is open in $X_s$. 
	By construction the restrictions of $g p_{si}|$ and $g' p_{sj}$ to $W'$ agree and define an equivariant map $g'' \colon W \to Z$. 
	It is now immediate from its definition that the tuple $(s,W, g'' \colon W \to Z)$ is an object of $((A_K)_Z)_{/f}$ that maps to both $(i,U,g)$ and $(j,V,g')$. This concludes the proof of (i).
	
	Now we turn to (ii). Note that in the proof of (i) we 
	have only needed the fact that the spaces $X_i$ are locally compact to use Lemma \ref{lem:nonemp-cofi}, namely to show that the domain of 
	a map $K_i \to Z'$ can be extended to an open subspace $K_i \subset U \subset X_i$. When $X_i$ is Tychonoff and $K_i$ is an orbit, we may derive the same extension property using Theorem \ref{thm:abels} in place of Lemma \ref{lem:nonemp-cofi}. Consequently (i) also holds for Tychonoff $G$-spaces whenever $K = Z \in \Orb$.
	Applying this fact to a constant cofiltered diagram with value a fixed Tychonoff space $X$, for all orbits $Z \subset X$ we 
	obtain an equivalence
	\[
	\colim_{Z \subset U} \colim_{U \to Z'} E(Z') \xto{\sim} \colim_{Z \to Z''} E(Z'').
	\]
	At last, observe that $((\Orb)_{Z/})^{\op}$ has a final object, namely $\id_Z$, and so the colimit above agrees with $E(Z)$. 
	\end{proof}
	
	\begin{prop} \label{prop:stalk-rep}
	Let $X$ be a Tychonoff $G$-space and $Z \subset X$ an orbit.
	For all $E \in \Psh(\Orb,C)$, the stalk of $\uE_X$ at $Z$ is given by $E(Z)$. 
	\end{prop}
	\begin{proof} We may equivalently compute the stalks of $\preuE_X$, 
	which are given by 
	\[
	(\preuE_X)_Z = \colim_{U\supset Z} \colim_{U \to Z'}E(Z').
	\]
	From here the result is immediate using Lemma \ref{lem:technical} (ii).
	\end{proof}
	
	Next we prove that Bredon cohomology satisfies cofiltered compact codescent.
	
	\begin{lem} \label{lem:profin}
	Let $C$ be a presentable category.
	If $X = \lim_{I \ni i} X_i$ is a cofiltered limit 
	in $\CHaus_G$ with projection maps $p_i \colon X \to X_i$,
	then for every coefficient system $E: \Orb^\op \to C$ the canonical map 
	\[
	\colim_{I^\op \ni i} p_i^*\uE_{X_i} \to\uE_X
	\]
	is an equivalence. 
	\end{lem}
	\begin{proof} The map above is induced upon sheafification from the map of 
	presheaves $\colim_{I^\op \ni i} p_i^\ast \preuE_{X_i} \to \preuE_X$. Furthermore, the presheaves  and $p_i^\ast \preuE_{X_i}$ for each $i \in I$ can be viewed as a restriction of functors out of $\mathcal{KO}(X)^{\op}$, given on objects by 
	\begin{equation}\label{map:preut}
	p_i^\ast\preutE_{X_i} (S)  = \colim_{I^\op \ni i} \colim_{p_i(S) \subset V} \colim_{V \to Z} E(Z), \qquad \preutE_X(T) = \colim_{T \to Z'} E(Z')
	\end{equation}
	where the colimits are ranging over the poset $I$, equivariant open subspaces of $X_i$ containing a given $p_i(S)$, $((\Orb)_{S/})^\op$ and $((\Orb)_{T/})^\op$ respectively. 
	Similarly, the map $\colim_{I^\op \ni i} p_i^\ast \preuE_{X_i} \to \preuE_X$ comes from the restriction of a natural transformation $\colim_{I^\op \ni i} p_i^\ast \preutE_{X_i} \to \preutE_X$. 
	In light of Lemma \ref{lem:ko}, we can thus equivalently prove that for each compact $K \subset X$ the map
	\begin{equation}
	\colim_{I^\op \ni i} \colim_{p_i(K) \subset V} \colim_{V \to Z} E(Z) \to \colim_{K \to Z'} E(Z'),
	\end{equation}
	obtained by evaluation at $K$, is an equivalence. This is precisely the content of Lemma \ref{lem:technical} (i).
	\end{proof}
	
	\begin{rmk}
	When $C=\An$, we may also prove Lemma \ref{lem:profin} by arguing that the comparison map $\colim_{I^\op \ni i} p_i^*\uE_{X_i} \to\uE_X$
	is compatible with taking colimits in the variable $E$ and hence we may assume that $E = \hom_{\Orb}(-,Z)$ is a representable and 
	thus $\uE_X$ agrees with the restriction of $\hom_{\GTop}(-,Z)$
	to $\cO_G(X)^\op$. With that reduction in place all the sheaves involved take values in $\Set$ and thus it is sufficient to check that the comparison map is a stalkwise isomorphism. The result now follows from Proposition \ref{prop:stalk-rep}. In fact, this proof doesn't require the spaces to be compact, so that one obtains a more general statement.
	\end{rmk}
	
	\begin{thm} \label{thm:breprofin}
	Let $C$ be a presentable category.
	For every $E \in \Psh(\Orb, C)$, Bredon homology with coefficients in $E$ satisfies cofiltered compact codescent.
	\end{thm}
	\begin{proof} Fix a compact Hausdorff space and a cofiltered limit $X = \lim_{I \ni i} X_i$ with projections $p_i \colon X \to X_i$ and transition maps $p_{ij} \colon X_i \to X_j$. Write $t \colon X \to \ast $ and $t_i \colon X_i \to \ast$ for the unique functions to the point.
		By Lemma \ref{lem:profin}, there is an equivalence
		\[
		\colim_{I^\op \ni i} p_i^* \uE_{X_i} \to \uE_X.
		\]
		Since $X$ is compact, we know that $t_* = t_!$ preserves colimits and thus
		\[
		t_* \uE_X \cong \colim_{I^\op \ni i} t_* p_i^* \uE_{X_i}
		\cong \colim_{I^\op \ni i} (t_i)_* (p_i)_* p_i^* \uE_{X_i}  
		\]
		By \cite{som}*{Lemma 2.5.10}, for all $i_0 \in I$ and $\cF\in \Shv(X_{i_0})$ we get that
		\[
		(p_{i_0})_* p_{i_0}^* \cF = \colim_{j \in I_{\le i_0}^\op} (p_{j,i_0})_* p_{j,i_0}^* \cF  
		\]
		Therefore, we obtain an equivalence
		\[
		t_* \uE_X \cong \colim_{I^\op \ni i} t_* p_i^* \uE_{X_i}
		\cong \colim_{I^\op \ni i} \colim_{j \in I_{\le i}^\op} (t_i)_*  (p_{j,i})_* p_{j,i}^* \uE_{X_i}  
		\]
		By cofinality of the diagonal map $I^\op \to \{(i,j) \in I^\op \times I^\op : j \le i\}$, this simplifies to
		\[
		t_* \uE_X \cong \colim_{I^\op \ni i} (t_i)_*\uE_{X_i}
		\]
		Since $t_*$ and $(t_i)_*$ compute global sections, this finally says that
		we have an equivalence
		\[
		\colim_{I^\op \ni i} \bre(X_i,E) \to \bre(X,E)
		\]
		as desired.
	\end{proof}
	
	We now have the following immediate corollary from cofiltered compact codescent: 
	\begin{coro} \label{coro:breclosed} 
	Let $C$ be a presentable category such that filtered colimits are left exact. For every $E \in \Psh(\Orb,C)$, Bredon cohomology with coefficients in $E$ satisfies closed descent for locally compact Hausdorff spaces. If $C$ is compactly assembled, then it is also $G$-homotopy invariant.
	\end{coro}
	\begin{proof} Since Bredon homology satisfies cofiltered compact codescent
	by Theorem \ref{thm:breprofin}, we are in position to apply Lemma \ref{lem:o+cc=>cl}. Homotopy invariance then follows by Proposition \ref{prop:htpyinv}.
	\end{proof}
	
	In fact, we even see that Bredon cohomology satisfies the strong form of closed descent, in which only one of the maps is a closed immersion and the other one is arbitrary, see Corollary \ref{coro:cl+}.
	
	\begin{lem} \label{lem:pullback-EX}
	Let $C$ be a presentable category, let $E \in \Psh(\Orb, C)$ and $X \in \GLCH$. 
		\begin{enumerate}[(i)]
			\item For any inclusion $j \colon Y \to X$ in $\GLCH$ where $Y$ carries the subspace topology
			and associated map $j_G \colon Y/G \to X/G$, there is an equivalence $j_G^* \uE_X \xto{\sim} \uE_Y$. 
			\item For each $G$-invariant open $U \subset X$ and induced maps $j \colon U \to X$, $i \colon (X\setminus U) \to X$,
			there is a cofibre sequence $i_G^! \uE_U \to \uE_X \to (j_G)_* \uE_{X \setminus U}$. 
		\end{enumerate} 
	\end{lem}
	\begin{proof}
		The equivalence \eqref{eq:profineq} applied to a constant $\N^{\op}$-indexed limit says in particular that for any $X \in \GLCH$ and any $G$-invariant compact $K \subset X$ we have an equivalence
		\begin{equation} \label{eq:colim-res}
			\colim_{K \subset U} \colim_{U \to Z} E(Z) \cong \colim_{n \ge 1} \colim_{K \subset U} \colim_{U \to Z} E(Z) \cong \colim_{K \to Z} E(Z).
		\end{equation}
		Now (i) follows from Lemma \ref{lem:ko}, and (ii) is a direct consequence of (i) by Proposition \ref{prop:Shv-cofibre}.
	\end{proof}
	
\begin{rmk}\label{rem_more}
	The previous statement is expected to remain valid beyond the setting of locally compact Hausdorff spaces, more precisely for Tychonoff spaces. Indeed, one should be able to establish the case $C = \An$ by reducing $E$ to representable objects and then working with stalks using Proposition \ref{prop:stalk-rep}. The general case would then follow by tensoring with an arbitrary $C$. As this extension is not required for our purposes, we do not pursue it here.
\end{rmk}
	
	\subsection{Agreement with Bredon cohomology}
	\label{subsec:bredoncw}
       We now compare Bredon sheaf cohomology of a $G$-space $X$ to classical singular Bredon cohomology. We follow the proof of Petersen in the non-equivariant case \cite{petersen}.

       Recall that singular Bredon cohomology
       for the target category $\Sp$ is defined by the mapping spectrum
         \[
         C_{\mathrm{Br}}^*(X,E) = \mathrm{hom}_{\Fun(\Orb^{\op},\Sp)} \bigl(\Sigma^\infty_+ \Sing(X^\bullet),\, E\bigr)
         \]
           where $\Sigma^\infty_+ \Sing(X^\bullet)$ denotes the functor $\Orb^{\op} \longrightarrow \Sp$ which sends $G/H$ to the suspension spectrum $\Sigma^\infty_+ \Sing(X^H)$ of the strict $H$--fixed points $X^H$. This mapping spectrum is also equivalent to the end
           \begin{equation}\label{Bredon_end}
            C_{\mathrm{Br}}^*(X,E)  = \int_{G/H \in \Orb} E(G/H)^{\Sing(X^H)} \ .
           \end{equation}
         If $E$ takes values in a presentable category $C$, then the last formula still makes sense and produces an object of $C$. That is the definition of singular Bredon cohomology with values in $C$. The category of \emph{$G$-anima} is the functor category
         
        \[
         \An^G := \Psh(\Orb) = \mathrm{Fun}(\Orb^\op, \An) \ .
        \]
        Every $G$-space $X$ clearly yields a $G$-anima as 
        $
        G/H \mapsto \Sing(X^H) .
        $ It is a result of Elmendorff \cite{elmendorf} that every $G$-anima arises this way from $G$-CW complexes.
        Clearly, the above definition \eqref{Bredon_end} of singular Bredon cohomology only uses the underlying $G$-anima of a $G$-space $X$ and therefore makes sense for arbitrary $G$-anima in place of $G/H \mapsto \Sing(X^H)$. 
        Defined in this way,  for fixed $E$, Bredon cohomology of a $G$-anima $X$ defines a functor 
         \[
         C_{\mathrm{Br}}^*(-, E): \quad (\An^G)^\op \to C
         \]
         which preserves limits. Conversely every limit preserving functor of this sort is determined by its restriction along Yoneda and thus induced by a functor $E: \Orb^\op \to C$. 
         
\begin{lem}\label{lem_desc} Let $C$ be a presentable category and $E \in \Psh(\Orb, C)$. Then the assignment 
\[
\GTop^\op \to C,  \qquad X \mapsto C_{\mathrm{Br}}^*(X,E) 
\]
is a hypersheaf with respect to the topology of $G$-invariant open covers. 
\end{lem}
\begin{proof}
Given the definition of Bredon sheaf cohomology as an end, it suffices to show that for each subgroup $H \subseteq G$ the functor $X \mapsto \Sing(X^H)$ is a hypercosheaf of anima on $\GTop$. This is turn follows from the fact that every $G$-invariant hypercover $U_\bullet$ of $X$ induces a hypercover on each fixed point set $U_\bullet^H \to X^H$ since we are simply restricting to the subspace topology on $X^H \subseteq X$. Thus the claim follows from the assertion that $\Sing$ as a functor $\Topp \to \An$ is a hypercosheaf, which in turn is proven in \cite{di}*{Theorem 1.3}. 
\end{proof}

\begin{prop}\label{prop_map}
Let $C$ be a presentable category and $E \in \Psh(\Orb, C)$. Then
there is a unique map 
\begin{equation}\label{map_bre}
\bre(X,E)\;\to\; C_{\mathrm{Br}}^*(X,E),
\end{equation}
natural in $X$, which is the identity when restricted to $\Orb$.
\end{prop}
\begin{proof}
Since the left hand side is $t^*(E)$ using the terminology of Definition \ref{def_bredon}, we get that since the target is a sheaf by the previous lemma \ref{lem_desc}, such a map is by adjunction determined by a transformation on orbits.
\end{proof}

Now we want to determine when the map is an equivalence. This will require an assumption similar to the assumption in the non-equivariant case, see e.g. \cite{ha}*{A.4} or \cite{petersen}. 

\begin{defn}
Let $C$ be a presentable category and $E \in \Psh(\Orb, C)$.
We say that a $G$-space $X$  is cohomologically contractible with respect to $E$  if for every orbit $Z \subseteq X$ the restriction map
\[
\colim_{U \supseteq Z } C_{\mathrm{Br}}^*(U,E) \to C_{\mathrm{Br}}^*(Z,E) = E(Z)
\]
is an equivalence, where the colimit ranges over all $G$-invariant open neighborhoods $U$ of the orbit $Z$. 
\end{defn}

\begin{ex}
Assume that $X$ is equivariantly sublocally contractible, that is for each orbit $Z \in X/G$ and each $G$-invariant open neighborhood $U \supseteq Z$ there exists a smaller equivariant neighborhood $V$ with $Z \subseteq V \subseteq U$ such that the inclusion $V \subseteq U$ is $G$-homotopic under $Z$ to a map that factors through the inclusion $Z \to U$. Then $X$ is clearly cohomologically contractible with respect to $E$ for every $E: \Orb^\op \to D$.

A further instance of that is if $X$ is equivariantly contractible in the sense that we can choose $V$ above, such that $Z \to V$ is already a $G$-homotopy equivalence. For example, this is the case if $X$ is a $G$-CW complex. 
\end{ex}

\begin{prop}\label{agreement}
Let $C$ be a compactly assembled category and $E \in \Psh(\Orb, C)$. Assume that $X$ is Tychonoff, cohomologically contractible with respect to $E$ and such that $X/G$ is hypercomplete. Then the map \eqref{map_bre} 
\[
\bre(X,E)\;\xto{\simeq}\; C_{\mathrm{Br}}^*(X,E).
\]
is an equivalence
\end{prop}

\begin{proof}
The map $\bre(X,E) \to C_{\mathrm{Br}}^*(X,E)$ can be considered as a map of sheaves on $X/G$. By naturality for every orbit $Z$ the triangle
\[
\begin{tikzcd}
	 (\underline{E}_X)_Z = \colim_{U \supseteq Z} \bre(U,E) \arrow{rr} \arrow{rd} & & \arrow{ld} \colim_{U \supseteq Z} C_{\mathrm{Br}}^*(U,E) \\
	 & E(Z) &		
\end{tikzcd}
\]
commutes. The left map is an equivalence by Proposition \ref{prop:stalk-rep} and the right hand map by the assumption that $X$ is  cohomologically contractible with respect to $E$. Thus the map of sheaves is an equivalence on stalks which by hypercompleteness of $X/G$ and the fact that $C$ is compactly assembled implies that it is an equivalence.
\end{proof}
	
\begin{coro}\label{corCW}
For every $G$-CW complex $X$, the map $\eqref{map_bre}$ is an equivalence.
\end{coro}
\begin{proof}
Clearly $G$-CW complexes are equivariantly locally contractible and Tychonoff. So it remains to check that $X/G$ is hypercomplete. But $X/G$ is a CW complex and every CW complex is hypercomplete. This is a well-known result, but we have not found a reference except the MathOverflow post by Marc Hoyois \cite{hoyois}, thus we record the argument of Hoyois here: any colimit of hypercomplete $\infty$-topoi is hypercomplete since hypercompletion is a localization of $\mathrm{RTop}$. For every CW complex $Y$, the topos of sheaves is the colimit of the topoi associated with the skeleta $Y_n$ by \cite{htt}*{Proposition 7.1.5.8} and those are hypercomplete by the fact that they are finite dimensional. 
\end{proof}

\begin{rmk}
We caution the reader that, in general, $G$-CW complexes are not locally compact; this property holds only in the locally finite case. Consequently, the previous statement (and its proof) necessarily requires working in the broader setting of $G$-topological spaces and not just locally compact $G$-spaces.

For locally finite $G$-CW complexes, one can also obtain a version of the preceding result by combining $G$-homotopy invariance with descent using a version of Elmendorf's theorem for locally finite $G$-CW complexes. However, formulating and proving such a version requires additional technical care, so we do not pursue it here.
\end{rmk}

	\section{Uniqueness of Bredon sheaf cohomology}
	
	We are now in position to identify the various categories of functors satisfying descent that we have considered 
	in terms of Bredon sheaf cohomology.
	
	\begin{thm} \label{thm:breeq}
	Let $C$ be a compactly assembled category. There is an equivalence of categories
		\[
		(-)|_{\Orb^{\op}} \colon \Fun^{o,cc}(\GLCH^{\op},C) \to \Fun(\Orb^\op,C)
		\]
		whose inverse is given by $E \mapsto \bre(-,E)$.
	\end{thm}
	\begin{proof}
	In light of Theorem \ref{thm:breprofin} and Corollary \ref{coro:breclosed}, the adjunction given in 
	\eqref{def:jstar} restricts to an adjunction 
	$\bre\colon \Fun(\Orb^\op,C) \longleftrightarrow \Fun^{o,cc}(\GLCH^{\op},C) \colon (-)|_{\Orb^{\op}}$.
	By Theorem \ref{thm:res-orb-cons} and Lemmas \ref{lem:bredon-counit} and \ref{lem:bredon-unit}, both the unit and counit of the adjunction are equivalences.
	\end{proof}
	
	In other words: Bredon sheaf cohomology is the unique functor $\GLCH^{\op} \to C$ satisfying open descent and cofiltered compact codescent. This is a very strong uniqueness result for cohomology theories and generalizes the non-equivariant case due to Clausen in unpublished work.

	\subsection{Compactly supported Bredon cohomology} \label{subsec:csupp}
		
		Recall that for a coefficient system $E \colon \Orb^\op \to D$ and a locally compact Hausdorff space with $G$-action $X \in \GLCH$ we have defined the sheaf $\uE_X \in \Shv(X/G, D)$, see the text after Definition \ref{def_bredon}. 
		
		\begin{defn}
		For a given dualizable category $D$ and $X \in \GLCH$, we define its compactly supported Bredon cohomology with coefficients in a functor $E \colon \Orb^\op \to D$ as 
		\[
		\brec(X,E) = \Gamma_c(X/G,\uE_X).
		\]
		\end{defn}

		\begin{prop}\label{prop:brec}
		This cohomology has the following properties: 

		\begin{enumerate}
		\item
		For every open inclusion $j: U \to X$ with closed complement $i: Z \to C$ we get an induced fibre sequence
		\[
		\brec(U,E) \to \brec(X,E) \to \brec(Z,E)
		\]
		\item
	        We have that 
		$
		\brec(X,E) = \fib(\bre(X^\infty,E) \to \bre(\infty,E))
		$
                 where $X \to X^\infty$ is the one point compactification. 
                 \item $\brec(X,E)$ refines to a functor on  $\GLCH^{pdp}$ and agrees with the value at $X$ of the extension of $\bre(-,E)$ to $\GLCH^{pdp}$
                (Theorem \ref{thm:pdp}).
                 \item $\brec(X,E)$ satisfies closed descent and open codescent.  
		\end{enumerate}
		\end{prop}
		\begin{proof}
		The first assertion follows from Lemma \ref{lem:pullback-EX} (ii). The second is a special case for $U = X \subseteq X^\infty$. The third then follows from Theorem \ref{thm:pdp} and  Theorem \ref{thm:uniqueness} below and the fourth from general properties of the extension, see Lemma \ref{lem:oc=cl} and 
	Lemma \ref{lem:codesc}.	\end{proof}	
	
	Again we have a strong uniqueness result for compactly Bredon supported sheaf cohomology.
	
		\begin{thm} \label{thm:uniqueness}
	For any dualizable category $D$ restriction induces equivalences 
	\[
	\Fun^{oe,cc}((\GLCH^{pdp})^{\op},D)
	\xto{\simeq} \Fun^{cc,cl}(\GCH^{\op},D) \xto{\simeq} \Fun(\Orb^\op,D). \]
	with inverse given by compactly supported Bredon sheaf cohomology. 
	\end{thm}
	\begin{proof}
	Follows immediately from Corollary \ref{cor_extension}, Theorem \ref{thm:pdp}, and Theorem \ref{thm:breeq}
	\end{proof}
	
	\section{Constructibility of Bredon sheaf cohomology} 
	
	A \emph{stratified space} over a poset $P$ is a continuous function $\pi \colon Y \to P$, where $P$ carries the Alexandroff topology. For each $p \in P$ write $Y_p = \pi^{-1}(p)$ for its fiber and $j_p \colon Y_p \hookrightarrow Y$ for the canonical inclusion. A sheaf $\cF\in \Shv(Y,C)$ with values on a presentable category $C$ is said to be \emph{constructible} if each restriction $\cF|_{Y_p} \defeq j_p^\ast \cF$ is locally constant. Write $\Shv_{c,P}(Y,C)$ for the full subcategory of $\Shv(Y,C)$ given by constructible sheaves with respect to the stratification $P$.	
	
	By Lemma \ref{lem:orbitproj}, if $X$ is a $G$-space then both $X$ and $X/G$ are canonically stratified over the poset $P_G$ of conjugation classes of subgroups of $G$.
	We shall only consider this stratification, called the stratification by \emph{orbit types}, and therefore we will omit $P_G$ from the notation and simply write $\Shv_{c}(X,C)$ and $\Shv_{c}(X/G,C)$ for constructible sheaves.  
	
	For a fixed $G$-space $X$ and any presentable category $C$, we have a functor 
	\begin{equation}\label{def:ufun}
		\underline{(-)}_X \colon \Psh(\Orb,C) \xto{\bre} \Shv(\GTop,C)
		\xto{\res} \Shv(X/G,C), \qquad E \mapsto \uE_X.
	\end{equation}
	We will concentrate in the case of $\An$-valued sheaves; one recovers the general definition of \eqref{def:ufun} upon tensoring by $C$.
	
	\begin{thm} \label{thm:EX-constr}
	For all $X \in \GLCH$ the associated sheaf $\uE_X$ is constructible.
	\end{thm}
	\begin{proof} Fix a conjugacy class $(H)$ of $G$. In light of Lemma \ref{lem:pullback-EX} it suffices to show that $\uE_{X_{(H)}}$ is locally constant; equivalently, for each orbit $Z \subset X_{(H)}$ we ought to see that there exists a $G$-invariant open $U \supset Z$ such that $\uE_{U\cap X_{(H)}}$ is constant.
		
	Given an orbit $Z \subset X_{(H)}$, by Theorem \ref{thm:abels}, we know that there exists a $G$-invariant open set $U \supset Z$ of $X$ such that $U$ is homeomorphic to $G \times_H V$, that is, it is induced from an $H$-invariant open subspace $V \subset U$. Furthermore, by Lemma \ref{lem:stabi-abels} we know that $V \cap X_{(H)} \subset U^H$, which in particular says that $U \cap X_{(H)}$ is homeomorphic to $G/H \times Y$ for some space $Y$ with trivial $G$-action. Consequently, we may without loss of generality assume that $X = G/H \times Y$ and prove that $\uE_{G/H \times Y}$ is the constant sheaf, which follows from Proposition \ref{lem:ind-res} and Example \ref{ex:bredon-trivial}.
	\end{proof}
	
	\begin{rmk}
	Theorem \ref{thm:EX-constr} should remain valid for Tychonoff spaces. The only place where locally compact Hausdorff is needed is in using Proposition \ref{lem:ind-res} which is also expected to be true more generally, see Remark \ref{rem_more}.
	\end{rmk}
	
	In light of Theorem \ref{thm:EX-constr}, we may corestrict \eqref{def:ufun} to a functor 
	\begin{equation}\label{def:uctbl}
	\underline{(-)}_X \colon \Psh(\Orb,\An) \to \Ctbl(X/G,\An).
	\end{equation}
	Since \eqref{def:ufun} preserves colimits since it is a composition of left adjoints and thus a left adjoint itself, it follows that \eqref{def:uctbl} is also cocontinuous.
	
	\subsection{Exit paths} \label{subsec:exit}
	
	Under suitable hypotheses, the category of constructible sheaves is a certain presheaf category:

	\begin{thm}[\cite{ha}*{Theorem A.9.3}] \label{thm:exitpath}
	 Let $P$ be a poset satisfying the ascending chain condition and $Y$ a paracompact space of locally singular shape. If $\pi \colon Y \to P$ is a conical stratification, then there exists an $\infty$-category $\Exit_P(Y)$ such that 
		\[
		\phi \colon \Shv_{c,P}(Y,\An) \xto{\sim} \Fun(\Exit_P(Y),\An).
		\]
		which restricts to an equivalence $\Shv_{c,P}(Y,\Set) \cong \Fun(\ho(\Exit_P(Y)),\Set)$.
	\end{thm}
	
	We refrain from expanding on the hypotheses of the theorem above; they apply to the orbit space of a smooth $G$-manifold with respect to the stratification of orbit-types (\cite{mayeda}*{Remark 4.1.7} and \cite{ayalaetal}*{Example 3.5.15}) and this is the only situation we will consider. 
	
	We will also only need an explicit understanding of $\exit_P(Y) \defeq \ho(\Exit_P(Y))$ and the equivalence above in the $\Set$-valued case, which we now recall. We refer to \cite{ha}*{Sections A.6 and A.9} for more details. 
	
	\begin{defn} \label{def:exitpath}
	Let $\pi \colon Y \to P$ be a stratified topological space. An \emph{exit path} in $Y$ is a path $\gamma \colon [0,1] \to Y$ such that $\pi(\gamma(0)) \le \pi(\gamma(t)) = \pi(\gamma(1))$ for all $t \in (0,1]$. A \emph{multiple-exit path} is a finite concatenation of exit paths. 
	A homotopy between multiple-exit paths $\gamma_0$ and $\gamma_1$ is a homotopy of paths $h \colon \gamma_0 \simeq \gamma_1$ such that $h_s$ is a multiple-exit path
	for all $s \in [0,1]$. The \emph{exit path $1$-category} $\exit_P(Y)$ has as objects
	the points of \(Y\) and homotopy classes of multiple-exit paths as morphisms. Composition is induced by usual concatenation of paths.
	\end{defn}
	
	\begin{rmk} 
	In \cite{mayeda}*{Definitions 2.1.4, 2.1.5 and 2.1.6} 
	multiple-exit paths are referred to as exit paths. What we call exit paths
	in Definition \ref{def:exitpath} are the $1$-simplices of the simplicial set $\Sing^P(Y)$
    introduced in \cite{ha}*{Definition A.6.2}, which under the hypotheses of Theorem \ref{thm:exitpath} becomes a model for $\Exit_P(Y)$.
    \end{rmk}
    
    \begin{rmk}
    We shall exploit the fact that any morphism in $\exit_P(Y)$ 
    is a finite composition of morphisms represented by exit paths. 
    The converse is not true in full generality: a composition 
    of two exit paths need not be homotopic to an exit path in a way that is 
    compatible with the stratification. This becomes true under
    the hypotheses of Theorem \ref{thm:exitpath}, which is a consequence of the 
    more general fact that $\Exit_P(Y)$ is an 
    $\infty$-category (\cite{ha}*{Theorem A.6.4}).
    \end{rmk}
	 
	Let $Y$ be a stratified space over a poset $P$ that lies in the hypotheses of Theorem \ref{thm:exitpath} and $\cF \in \Shv_{c,P}(Y,\Set)$. We now explain how the functor $\phi$ of Theorem \ref{thm:exitpath} associates to $\cF$ a functor $\phi(\cF) \colon \exit_P(Y) \to \Set$. On objects, it maps $y \in Y$ to the stalks
	of $\cF$ at $y$, that is $\phi(\cF)(y) \defeq \cF_y = \colim_{U\ni y} \cF(U)$. If $\gamma \colon [0,1] 
	\to Y$ is an exit path from $y$ to $y'$, then $\phi(\cF)([\gamma])$ is the \emph{transport map} $\cF_y \to \cF_{y'}$ along $\gamma$, which we proceed to describe.
	
	Given $\eta \in \cF_y = \colim_{U \in y}\cF(U)$, we may consider an open set $U \ni y$ such that there exists $\widehat{\eta} \in \cF(U)$ representing $\eta$.
	By continuity of $\gamma$, there exists $t > 0$ such that $\gamma([0,t]) \subset U$, 
	and thus there exists a map $\cF(U) \to \cF_{\gamma(t)}$. Since $\gamma$ is an exit path, it follows that $\gamma|_{[t,1]}^\ast \cF$ is a locally constant on $[t,1]$ and hence constant. In particular
	there is a zig-zag of bijections
	\[
	\cF_{\gamma(t)} \xleftarrow{\sim} (\gamma^\ast\cF)([t,1]) \xrightarrow{\sim} \cF_{\gamma(1)}
	\]
	The image of $\eta$ under the transport map is the image of $\widehat{\eta}$
	under the composition
	\[
	\cF(U) \to \cF_{\gamma(t)} \xto{\sim} \cF([t,1]) \xto{\sim} \cF_{\gamma(1)} = \cF_{y'}.
	\]
	One checks that this is independent of the choices of $\widehat{\eta}$ and $t$ and thus yields a well defined map.
	
	\subsection{Mayeda's functor}
	In \cite{mayeda}, Mayeda proves that for any smooth $G$-manifold $M$ equipped with its stratification of orbit types the functor $\exit(M) \to \exit(M/G)$ is a right fibration: for every exit path $\gamma \colon I \to M/G$ and point $y$ in the orbit $\gamma(1)$, we have an exit path $\widetilde{\gamma}_y \colon I \to M$ lifting $\gamma$ and satisfying $\widetilde{\gamma}_y(1) = y$. 
	Using this, they consider a functor
	\begin{equation}\label{def:mayeda}
		m \colon \exit(M/G) \to \Orb^\op
	\end{equation}
	mapping a point in $M/G$ to the orbit it represents in $\Orb$, and an exit path $\gamma$ from $Z$ to $Z'$ to the function
	\[
	m(\gamma) \colon Z' \to Z, \quad m(\gamma)(y) \defeq \widetilde{\gamma}_y(0).
	\]
	Fix a smooth $G$-manifold $M$. In light of Theorem \ref{thm:exitpath}, we have a functor
	\begin{equation}\label{def:uexit}
		\phi \circ \underline{(-)}_M \colon \Psh(\Orb,\An) \to \Ctbl(M/G,\An) \xto{\sim} \Fun(\Exit(M/G),\An).
	\end{equation}
	
	To conclude the section, we prove that for any presheaf $E \colon \Orb^\op \to \An$ the constructible sheaf $\underline{E}_M$ is classified by the composition 
	\[
	\Exit(M/G) \xto{m} \Orb^\op \xto{E} \An.
	\] 
	In other words:	
	\begin{thm}\label{thm:mayeda}
	The functor \eqref{def:uexit} is given by precomposition the functor \eqref{def:mayeda} considered in \cite{mayeda}.
	\end{thm}
	\begin{proof} 
		
	Write $q \colon M \to M/G$ for the quotient map and $\hom_G \defeq \hom_{\GLCH}$.
	Since $\phi \circ \underline{(-)}_X$ is cocontinuous, it suffices to study its restriction 
	to $\Orb$ along the Yoneda embedding $y \colon \Orb\to\Psh(\Orb)$ and show that in coincides with $m^\ast \circ y$. The latter corresponds to the bifunctor 
	\[
	\Orb \times \exit(M/G) \xto{1 \times m} \Orb \times \Orb^\op \xto{\hom_{\Orb}} \Set \subset \An.
	\]
	Now we study the corresponding identification of $\phi \circ \underline{(-)}_X$. Since representables are preserved by left Kan extension and sheafification commutes with the restriction $\Psh(\LCHaus) \to \Psh(X/G)$, a representable $y(Z)$ is mapped to the sheafification of
	\[
	\rho_Z(U) = \hom_{G}(q^{-1}(U),Z),
	\] 
	which is already a sheaf. Hence $\phi \circ \underline{(-)}_X \circ y$ corresponds to the bifunctor
	\[
	b \colon \Orb \times \exit(M/G) \to \Set \subset \An
	\]
	which maps $(Z,O) \xto{(g,\gamma)} (Z',O)$ to
	\[
	\phi(\rho_Z)(O) \xto{\phi(g_*)_O} \phi(\rho_{Z'})(O)
	\xto{\phi(\rho_Z)(\gamma)} \phi(\rho_{Z'})(O').
	\] 
	By Proposition \ref{prop:stalk-rep}, the stalk of $\rho_Z$ at an orbit 
	$O$ is 
	\[
	\phi(\rho_Z)(O)= (\rho_Z)_O = \hom_{G}(O,Z).
	\] 
	To conclude the proof, we want to show that $\phi(\rho_Z)(\gamma)$
	agrees with $m(\gamma)^\ast$.
	Notice that by varying the orbit $Z$,
	the maps $\phi(\rho_Z)(\gamma)$ define in fact a natural transformation
	$\hom(O,-) \Rightarrow \hom(O',-)$, and so by the Yoneda lemma 
	it has be identified with precomposition of a function; namely, the image of $\id_O \in (\rho_O)_O$ under the transport map $(\rho_O)_O \to (\rho_O)_{O'}$ along the exit path $\gamma$. 
	
	We now explain how the transport map acts on $\id_O$, following the description given in Section \ref{subsec:exit}. Put $O_t \defeq \gamma(t)$. First we lift $\id_Z \in (\rho_O)_O$ to an element $f \in \rho_O(U)=\hom_{G}(U,O)$ for some invariant open subset \(U\) of \(X\), such that $q(U)$
	contains a path $\gamma \vert_{[0,t]}$ for some $t > 0$. Note that there indeed exists such a \(t\) by the continuity of \(\gamma\). There is thus a map $\rho_O(U) \to (\rho_O)_{O_t}$ given by restriction along $O_t \hookrightarrow U$. Since $\gamma$ is an exit path, we know that $\gamma|_{[t,1]}^\ast \rho_O$ is (locally) constant. Hence we have a zig-zag of restriction-induced bijections
	\[
	\hom(O_t,O) \xleftarrow{\sim} \colim_{\gamma([t,1]) \subset q(V)}\hom_{G}(V,O) \xrightarrow{\sim} \hom_G(O_1,O) = \hom_G(O',O).
	\]
	In particular, this says that $f|_{O_t} \colon O_t \to O$ can be extended to map $g \colon V \to O$ where $V$ is an invariant open such that $q(V)$ contains $\gamma([t,1])$, and the image of $\id_Z$ along the transport map is the restriction of $g$ to $O'$:
	\[\begin{tikzcd}[ampersand replacement=\&]
	{O'} \& V \& \\
	{O_t} \& U \& O \\
	\& O
	\arrow[hook, from=1-1, to=1-2]
	\arrow["{(\phi(\rho_O)(\gamma))(\id_O)}",bend left  = 50, from=1-1, to=2-3]
	\arrow["g", from=1-2, to=2-3]
	\arrow[hook, from=2-1, to=1-2]
	\arrow[hook, from=2-1, to=2-2]
	\arrow["f", from=2-2, to=2-3]
	\arrow[hook, from=3-2, to=2-2]
	\arrow[equals, from=3-2, to=2-3]
	\end{tikzcd}\]
	Hence we must see that the restriction of $g$ to $O'$ agrees wth $m(\gamma)$. Since maps 
	between orbits that agree on a point are equal, we only need to check that for a fixed, arbitrary $x' \in O$ we have $m(\gamma)(x') = g(x')$. Recall that the point $m(\gamma)(x')$ is defined by taking a lift $\widetilde{\gamma}^{x'} \colon I \to M$ of $\gamma$ satisfying $\widetilde{\gamma}^{x'}(1) = x'$ and setting $m(\gamma)(x') = \widetilde{\gamma}^{x'}(0)$.
	
	Since $q(U) \supset \gamma([0,t])$ and $q(V) \supset \gamma([t,1])$, and both $U$ and $V$ are $G$-invariant opens, it follows that
	$\widetilde{\gamma}^{x'}([0,t]) \subset U$ and $\widetilde{\gamma}^{x'}([t,1]) \subset V$.
	We may thus consider the restrictions $\gamma' = \widetilde{\gamma}^{x'}|_{[0,t]} \colon [0,t] \to U$ and $\gamma'' = \widetilde{\gamma}^{x'}|_{[t,1]} \colon [t,1] \to V$, which, since $O$ is discrete and $I$ is connected, imply that $f\gamma'$ and $g\gamma''$ are constant. Using the latter, the fact that $g$ and $f$ agree on $O_t$, and that $f|_O = \id_O$, we finally get
	\[
	g(x) = g(\gamma'(1)) = g(\gamma'(t)) = f(\gamma''(t)) = f(\gamma''(0)) = \gamma''(0)
	= \widetilde{\gamma}^{x'}(0) = m(\gamma)(x'),
	\]
	concluding the proof.
	\end{proof}
	
	\begin{coro}
	For every $G$-manifold $M$, Bredon cohomology  is given by the limit:
	\[
	\lim_{\Exit(M/G)} E(m(-))
	\]
	\end{coro}
	\begin{proof}
	Constructible sheaves are equivalent to $\Fun(\Exit(M/G), C)$. The composite functor 
\[ \Fun(\Exit(M/G), C) \to \Shv(M/G, C) \xto{\Gamma} C \]
	is right adjoint to the constant functor $C \to \Fun(\Exit(M/G), C)$, which is the constant sheaf functor which happens to land in constructible sheaves. Thus this composite is given by the limit over $\Exit(M/G)$.
	\end{proof}
	The previous result appears to be new, although it is implicitly contained in the work of Henriques \cite{henriques}. It concerns classical Bredon cohomology, since in the present setting the two notions agree, i.e.
$
\bre(M,E) = C_{\mathrm{Br}}^*(M,E),
$
see Proposition  \ref{agreement}. Concretely, the result shows that Bredon cohomology can be reconstructed from the exit-path $\infty$-category associated to the stratification of $M/G$. We expect that the statement extends more generally, for instance to $G$-CW complexes.
	\section{K-theory of equivariant sheaves and functions} \label{sec:appl}
	
	The goal of this section is to use the structural results of the previous sections to compute equivariant algebra and topological K-theory of categories of sheaves and $C^*$-algebras of continuous functions on a locally compact $G$-space. 
	
	\subsection{Localizing invariants of equivariant sheaves}
	
	In this section we compute localising 
	invariants associated to the category of equivariant sheaves
	on a locally compact Hausdorff $G$-space.  First we give a
	brief recollection of the relevant definitions.
	
	\begin{defn} \label{defn:locinv}
	A \emph{localising invariant}
	with values on a stable category $D$ is a functor
	$F \colon \Catdual \to D$ that maps $0$ to $0$ and Verdier sequences to cofibre sequences. It is \emph{finitary} if it preserves filtered colimits.
	\end{defn}
	
	The prime example of a finitary localising invariant is 
	the nonconnective $K$-theory functor
	\[
	K \colon \Catdual \to \Sp,
	\] 
	see \cites{bgt,efiloc}. Another prominent example is topological Hochschild homology (see e.g. \cite{som}*{Proposition 3.5.11}).
	By replacing $\Catdual$ by $\Catdual^G$ in Definition \ref{defn:locinv}, one arrives at the notion localising $G$-invariant, which recovers the above when $G=1$. 
	There is a universal localising invariant whose target is the category of \emph{noncommutative $G$-motives}
	\[
	M_G \colon \Catdual^G \to \NcMot_G.
	\]
	We refer to \cites{bgt, efiloc, efirig, rsw} for further details.

	Our computation relies on functoriality of sheaves on locally compact Hausdorff spaces and partially defined proper maps:
		
	\begin{prop} 
		\label{thm:shvxfunc}
		Let $C$ be a dualizable category. There is a functor
		\[
		\Shv(-,C) \colon (\LCHaus^{pdp})^\op \to \Catdual 
		\]
		which specializes to
		\[
		\Shv^\ast(-,C) \colon \LCHp^{\op} \to \Catdual, \qquad X \mapsto \Shv(X,D), \qquad f \mapsto f^\ast. 
		\]
		and
		\[
		\Shv_!(-,C) \colon \LCHopen \to \Catdual, \qquad X \mapsto \Shv(X,D), \qquad f \mapsto f_!. 
		\] 
		Moreover, the functor $\Shv$ maps open-closed sequences
		to Verdier sequences and cofiltered limits in $\CHaus$ to colimits in $\Catdual$.
		\qed
	\end{prop}
	\begin{proof} 	
	This is a conjunction of Proposition \ref{prop:Shv-cofibre}, \cite{volpe}*{Remark 6.17}, and \cite{som}*{Proposition 3.6.4, Corollary 3.6.5, and Proposition 3.6.7}.
	\end{proof}
	
	\begin{defn}
	We consider the functor
	\[
	\Shv(-, C) \colon (\GLCH^{pdp})^\op \to 
	\Catdual^G 
	\]
	induced by taking $G$-objects for both sides of the functor in Proposition \ref{thm:shvxfunc}.
	and
	\[
	\Shv_G(-,C)  \colon (\GLCH^{pdp})^\op \to 
	\Catdual^G \xto{\colim_{BG}} \Catdual.
	\]	
	\end{defn}
	
	\begin{rmk}\label{rmk:shvg=shvg} Note that 
	\begin{align*}
	\Shv_G(X,C) & = \colim^{\Catdual}_{\Delta^{op} \ni n} \Shv(G^n \times X,C)  \\
	&  = \lim^{\Cat_\infty}_{\Delta \ni n} \Shv(G^n \times X,C) = \Shv(X,C)^{hG}
	\end{align*}
	is equivalent to the usual definition of $G$-equivariant sheaves on $X$ with values in $C$. We write $C^{B}$ for the restriction of $\Shv_G(-,C)$ to $\Orb$. We have an equivalence $\Shv_G(G/H,C) \cong C^{BH}$ where the functoriality of the right hand side is described by the functor
	\[
	(-)_{hG}: \Orb \to \An \qquad G/H \mapsto BH \ .
	\]
	Note that maps of orbits $G/H \to G/H'$ induce covering maps $BH \to BH'$ (in particular injective on $\pi_1$). That is why the restriction functor
	\[
	C^{BH'} \to C^{BH}
	\]
	is strongly continuous. This of course also follows from the identification with equivariant sheaves and the properness of orbit maps.	In the sense of $G$-category theory, the functor $C^B$ is the Borel $G$-category associated with the category $C$. 
	\end{rmk}
	
	\begin{thm} \label{thm:shvmot}
	Let $G$ be a finite group and $D$ a dualizable category. For all finitary localising invariants $F \colon \Catdual \to D$, we have
	an equivalence
	\[
	F(\Shv_G(X,C)) \cong \brec(X,F(C^B))
	\]
	for all $X \in \GLCH$. 
	Similarly, for all localising $G$-invariants $H \colon \Catdual^G \to D$ we have 
	\[
	H({\Shv}(X)) \cong \brec(X,H({C^B})).
	\]
	\end{thm}
	\begin{proof}

	In view of Theorem \ref{thm_unique} and the fact that $F$ and $H$ preserve filtered colimits and map Verdier sequences to cofibre sequences, it suffices to show that $\Shv_G$ map cofiltered limits in $\CHaus_G$ to colimits, and
	open-closed sequences to Verdier sequences. 
	
	From Theorem \ref{thm:shvxfunc}, the fact that limits and colimits in functor categories are computed pointwise, and that $\colim_{BG}$ commutes with colimits, we are only left with showing that $\Shv_G$ maps open-closed sequences to Verdier sequences. As we have already observed, it already preserves cofibre sequences, so it would suffice to show that it maps (strongly continuous) fully faithful functors to (strongly continuous) fully faithful functors. The latter follows from the fact that in the present case fully faithfulness can be expressed in terms of counits, 
	and adjunctions between $G$-dualizable categories promote automatically to adjunctions in the $(\infty,2)$-category $\Catdual^G$; see e.g. \cite{hauglax}*{Theorem 4.6}.
	\end{proof}
	
	Finally we want to specialize Theorem  \ref{thm:shvmot}  to the case of the category of noncommutative ($G$-)motives, which is dualizable by 
	\cite{efirig}*{Theorem 3.1} so that we have an equivalence
	\[
	M_G(\Shv(X, C)) = \brec(X, M_G(C^B)) 
	\]
	One can make $\NcMot_G$ into a genuine $G$-category (i.e. a presheaf of categories on $\Orb$), which in particular implies that it is powered and tensored over $G$-anima (i.e. $\Psh(\Orb)$), see \cite{MaximeKaif} for details. The powering of a $G$-motive $M_G(C)$ with respect to a $G$-anima $X$ is then precisely the Bredon cohomology $C_{\mathrm{Br}}^*(X, M_G(C^B))$, so that we get using Corollary \ref{corCW}:
	
\begin{coro}\label{cor_genuine}
For $X$ a finite $G$-CW complex we have
\[
M_G(\Shv(X, C)) \simeq M_G(C)^{\uSing(X)} \ .
\]
where $\uSing$ is the underlying singular  $G$-anima $G/H \mapsto \Sing(X^H)$ of $X$ and the power is in the sense of genuine $G$-categories. 
\end{coro}
As a result one gets similar maps, whenever one has a $G$-functor from the $G$-category $\NcMot_G$ to some other $G$-category that preserves $G$-limits.

	\subsection{Equivariant \(E\)-theory of functions}
In analogy with \(G\)-motives, we now describe the equivariant \(E\)-theory \cites{connes1990deformations,guentner2000equivariant} of continuous functions on a locally compact Hausdorff space.  For this, we essentially only require the following universal property:
	
\begin{thm}[\cite{bunke2024theory}]\label{thm:E}
	There is a functor 
	\[
	e^G \colon \GCsAlg \to \E^G
	\] into a dualisable category that is:
	\begin{enumerate}
	\item equivariantly homotopy invariant;
	\item stable with respect to \(\cK_G = K(L^2(G) \otimes \ell^2)\); see \cite{bunke2024theory}*{Remark 3.14} for a precise definition;
	\item excisive, that is, it sends
	a short exact sequence of \(G\)-\(C^*\)-algebras to a fibre sequence;
	\item filtered colimit-preserving.
	\end{enumerate}
Furthermore, the functor \(e^G\) is the initial functor into a cocomplete stable category with these properties. In other words, denoting by $\mathrm{Fun}^{1,2,3,4}(\GCsAlg, D)\) functors into a cocomplete, stable category, restriction along \(e^G\) induces an equivalence \[\mathrm{Fun}^{\mathrm{colim}}(\E^G, D) \to \mathrm{Fun}^{1,2,3,4}(\GCsAlg, D),\] where the left hand side denotes colimit-preserving functors. 
	\end{thm}
	
	\begin{proof}
	By \cite{bunke2024theory}*{Proposition 3.55}, there is a functor \(e_{\sep}^G \colon \GCsAlg_{\sep} \to \E_{\sep}^G\) that is equivariantly homotopy invariant, \(\cK_G\)-stable, Schochet exact and countable filtered colimit-preserving, and is the initial functor with stable, countably cocomplete target categories with these properties. Combining with \cite{bunke2024theory}*{Theorem 3.58, 3.59}, we see that \(e_{\sep}^G\) is the initial functor into a countably cocomplete category that is equivariant homotopy invariant, \(\cK_G\)-stable, excisive and countable filtered colimit-preserving. Now by \cite{kerodon}*{Corollaries 9.3.5.27 06N9, 9.3.6.10 0694 and 9.3.6.11 0695}, \[e_G = \mathrm{Ind}_{\aleph_1}(e_{\sep}^G) \colon  \GCsAlg = \Ind_{\aleph_1}(\GCsAlg_{\sep}) \to \Ind_{\aleph_1}(\E_{\sep}^G) =: \E^G\] preserves filtered colimits, so that restriction along \(e_G\) induces an equivalence \[\mathrm{Fun}^{\colim}(\E^G, D) \to \mathrm{Fun}^{1,2,3,4}(\GCsAlg, D)\] for all \(D\) cocomplete and stable. The dualisability of \(\E^G\) is \cite{bunke2024theory}*{Theorem 1.1}. 
	\end{proof}
	
Continuing the analogy with sheaves, the analogue of Proposition \ref{thm:shvxfunc} is the following:

\begin{thm}
There is a functor 
\[
C_0(-) \colon (\LCHaus^{pdp})^\op \to C^\ast\mathrm{Alg}
\] that specialises to a contravariant functor \[(\LCHp)^\op \to C^\ast\mathrm{Alg}, \quad X \mapsto C_0(X), \quad f \colon X \to Y \mapsto f^* \colon C_0(Y) \to C_0(X)\] and a covariant functor \[\LCHopen \to \GCsAlg, \quad X \mapsto C_0(X)\] that takes an open embedding \(f \colon U \to X\) to \(C_0(U) \to C_0(X)\) by restricting along the collapse map \(X^+  \to U^+\) in \(\mathrm{CHaus}_*\).  
\end{thm}

\begin{proof}
We restrict the equivalence of categories \(\mathrm{LCHaus}^{pdp} \cong \mathrm{CHaus}_*\)  \cite{bunke2021lecture}*{Lemma 5.2} to \(\LCHp^\op\) and \(\LCHopen\). Finally by \cite{bunke2021lecture}*{Corollary 5.4}, the required functor \((\LCHaus^{pdp})^\op \to \GCsAlg\) is the functor \[C_0 : X \mapsto C_0(X) = \{f: X^+ \to \C \text{ continuous } : f(\infty) = 0\}\] implementing the Gelfand duality. 
\end{proof}

We now compose with the canonical functor \(e^G \colon \GCsAlg \to \E^G\) to get a functor 
\begin{equation}\label{eq-E}
	 e^G(C_0(-)) \colon (\GLCH^{pdp})^\op \xto{C_0(-)} \GCsAlg \xto{e^G} \E^G
\end{equation}
into equivariant \(E\)-theory.

	 \begin{prop}\label{prop:cts-sheaf}
	 The functor in \eqref{eq-E} satisfies cofiltered compact codescent and open-closed excision.
	 \end{prop}
	 
	 \begin{proof}
		Let \(X\) be a locally compact Hausdorff \(G\)-space, \(U \subseteq X\) an open $G$-equivariant subspace, and \(Z = X \setminus U\). The map \(Z \to X \to U\) in \(\GLCH^{pdp}\) induces an extension of \(G\)-\(C^\ast\)-algebras \[C_0(U) \to C_0(X) \to C_0(Z),\] which gets mapped to a fibre-sequence in \(\E^G\) by excision. Cofiltered compact codescent is clear as the universal functor \(\GCsAlg \to \E^G\) preserves filtered colimits by Theorem \ref{thm:E}.   	 
	 \end{proof}
	 
	 As a consequence of Theorem \ref{thm:uniqueness} we have:
	 
	\begin{coro}
	For any \(X \in \GLCH\), we have an equivalence \[e^G(C_0(X)) \simeq \brec(X, e^G(C_0(-)) \vert_{\Orb})\] in \(\E^G\).
	\qedhere
	\end{coro}

Consider the crossed product functor \(G \ltimes - \colon \GCsAlg \to \CsAlg\), taking a \(G\)-\(C^\ast\)-algebra to its (maximal) crossed product. In what follows, let \(\E\) denote the equivariant \(E\)-theory functor of Theorem \ref{thm:E} for the trivial group. We first record the following:

\begin{lem}\label{lem:descent}
The crossed product functor descends to a colimit-preserving functor \(G \ltimes - \colon \E^G \to \E\). 
\end{lem}

\begin{proof}
By \cite{bunkenoncomm2}*{Corollary 3.10, Lemma 3.32}, the crossed product functor is homotopy invariant and \(\cK_G\)-stable. By \cite{bunke2020non} it preserves extensions, and by \cite{bunke2021stable}*{Lemma 4.15} filtered colimits. The conclusion now follows from Theorem  \ref{thm:E}.
\end{proof}

Recall that \(E\) is presentably symmetric monoidal with respect to the maximal tensor product of \(C^*\)-algebras, with tensor unit given by the image of \(\C\). As a consequence, \(\mathrm{KU} \defeq \E(\C,\C)\) is a commutative ring spectrum, and \(\E\) has a \(\mathrm{KU}\)-linear structure. Denote by \(\mathrm{K}^{\mathrm{top}} \defeq \E(\C,-)  \colon \E \to \mathrm{Mod}_{\mathrm{KU}}\) the \emph{complex topological \(K\)-theory} functor. Post-composing the crossed product functor with the functor \eqref{eq-E}, we define \[\mathrm{K}^{\mathrm{top}}(G \ltimes C_0(-)) \colon (\GLCH^{pdp})^\op \xto{e_G(C_0(-))} \E^G \xto{G \ltimes -} E \xto{E(\C,-)}  \mathrm{Mod}_{\mathrm{KU}}.\] Let \(\mathrm{K}_G\) denote the restriction of \(\mathrm{K}^{\mathrm{top}}(G \ltimes C_0(-))\) to \(\Orb(G)^\op\). Note that when \(G\) is trivial, \(\mathrm{K}_G = \mathrm{KU}\).

\begin{coro}\label{coro:K-theory-crossed}
The functor \(\mathrm{K}^{\mathrm{top}}(G \ltimes C_0(-))\) satisfies cofiltered compact codescent and open-closed excision. Consequently, we have an equivalence \[\mathrm{K}^{\mathrm{top}}(G \ltimes C_0(X)) \simeq  \brec(X,\mathrm{K}_G)\] for any \(X \in \GLCH\). In particular, we have \(\mathrm{K}^{\mathrm{top}}(C_0(X)) \simeq \Gamma_c(X, \mathrm{KU})\).
\end{coro}
\begin{proof}
We have already seen in Proposition \ref{prop:cts-sheaf} that the functor \(e^G(C_0(-))\) satisfies cofiltered compact codescent and open-closed excision. This property is preserved by post-composition with the crossed product functor as the latter preserves filtered colimits and fibre-sequences by Lemma \ref{lem:descent}. It is further preserved by composing with topological \(K\)-theory \(\mathrm{K}^{\mathrm{top}} = \E(\C,-)\), which is excisive and preserves filtered colimits. The conclusion now follows from Theorem \ref{thm:uniqueness}.
\end{proof}	

\begin{rmk}
We remark that the agreement between complexified topological \(K\)-theory and sheaf cohomology for (second countable) locally compact Hausdorff spaces was already known via the Chern character from complexified \(K\)-theory to local cyclic homology, and the agreement of the latter with compactly supported sheaf cohomology by \cite{puschnigg2003diffeotopy}*{Theorem 8.6}. To the best of our knowledge, the agreement of topological \(K\)-theory with compactly supported cohomology with coefficients in \(\mathrm{KU}\) observed in Corollary \ref{coro:K-theory-crossed} has never explicitly been spelled out. 
\end{rmk}

	\section{The equivariant shape}\label{sec:shape}

We note that the whole construction of Bredon sheaf cohomology hinged on the adjunction
\[
t^*: \Psh(\Orb) \leftrightarrows \Shv(\GTop): t_*
\]
This adjunction exists for (pre)sheaves with values in any category, in particular also with values in anima in which case $\Psh(\Orb) = \An^G$. This is the case we consider now.

\begin{prop}
$t_*$ is a geometric morphism of topoi , that is $t^*$ preserves finite limits. 
\end{prop}

Note that for size issues $\Shv(\GTop)$ is not quite a topos (it is too large), but that is not the relevant point here, the second part of the statement makes sense as it stands. 

\begin{proof}
We factor the morphisms of sites $t$ as 
\[
\Orb \to \mathrm{Fin}_G \xto{t'} \GTop
\]
where $\mathrm{Fin}_G$ denotes the category of finite $G$-sets which is made into a site by considering the disjoint union Grothendieck topology, i.e. coverings are given by jointly disjoint, surjective injections. We clearly have that the first morphism $\Shv(\Orb) \to \Shv(\mathrm{Fin}_G)$ induces an equivalence. Therefore it suffices to check that $(t')^*$ preserves finite limits. This in turn reduces to verifying that finite limits of representable sheaves are preserved which follows since $t'$ preserves finite limits.
\end{proof}

\begin{coro} 
The functor $t^*$ has a pro left adjoint, that is there is a functor
\[
t_\natural: \Shv(\GTop) \to \mathrm{Pro}(\An^G)
\]
such that $\mathrm{Hom}(t_\natural(F), E) \simeq \mathrm{Hom}(F, t^*E)$ for $F \in \Shv(\LCHaus_G) $ and $E \in \An^G$. \qed
\end{coro}
	
\begin{defn}
The equivariant shape is the functor
\[
\underline{\Pi}_\infty: \GTop \to \mathrm{Pro}(\An^G) 
\]
obtained as the composition of the functor $t_\natural$ with the Yoneda embedding $\GTop \to \Shv(\GTop)$.
\end{defn}

Note that the equivariant shape $\underline{\Pi}_\infty X$ of the $G$-space $X$, which is a pro-$G$-anima, has the property that the underlying pro-anima is in fact the underlying pro-anima,  i.e. its restriction to the free $G$-orbit $G/e$, is the usual shape ${\Pi}_\infty X$ of $X$.  

\begin{rmk}
We note that the equivarant shape, as well as Bredon sheaf cohomology, of course make sense for arbitrary $G$-spaces, not just locally compact Hausdorff ones with the exact same definition. We just restrict to the latter one in this paper for the uniqueness statements. 

We can also think of the equivariant shape in terms of a relative shape of a topos, namely we can consider the slice topos $\Shv(\LCHaus_G)_{/X}$ or some small version of it. This comes with a geometric morphism to $\An^G$ and the shape $\underline{\Pi}_\infty(X)$ is the relative shape for this geometric morphism. 
\end{rmk}

Recall that for a $G$-anima $X$ and a functor $E: \Orb^\op \to D$ we have Bredon cohomology
$C_{\mathrm{Br}}^*(X, E) \in D$, see Section \ref{subsec:bredoncw}. This naturally extends by cofiltered limit extension to a functor
\[
\left(\mathrm{Pro}(\An^G)\right)^\op \to D \ .
\]
Concretely we have $C_{\mathrm{Br}}^*(``\underleftarrow{\mathrm{lim}}"X_i, E) := \underrightarrow{\colim} C_{\mathrm{Br}}^*(X_i, E)$ for a pro $G$-anima $``\underleftarrow{\mathrm{lim}}"X_i$.
\begin{prop}\label{prop_equiv}
The equivariant shape has the following properties:
\begin{enumerate}
\item For every $E: \Orb^\op \to D$ with $D$ compactly assembled there is a natural equivalence 
\[
\bre(X, E) = C_{\mathrm{Br}}^*(\underline{\Pi}_\infty X,E),
\]
that is Bredon sheaf cohomology of $X$ agrees with singular Bredon cohomology of the shape. 

\item The first property, if it holds for $D = \An$, uniquely characterises the equivariant shape. 

\item The functor $\underline{\Pi}_\infty$ satisfies open codescent, closed codescent and cofiltered compact descent.

\item There is a natural (in $X$) map $\uSing(X) \to \underline{\Pi}_\infty(X)$
for every $G$-space $X$, where  $\uSing(X)$ is the singular $G$-anima 
\[
G/H  \mapsto \Sing(X^H) \ .
\]
considered as constant pro object. It is an equivalence if $X$ is Tychonoff, sublocally contractible and such that $X/G$ is hypercomplete. 
\end{enumerate}
\end{prop}
\begin{proof}
For (1), we first treat the case where the target is $\An$. In this situation, adjunction yields
\[
C_{\mathrm{Br}}^*(\underline{\Pi}_\infty X,E)
= \mathrm{Hom}_{\mathrm{Pro}(\An^G)}\bigl(t_\natural(\underline{X}),E\bigr)
= \mathrm{Hom}_{\Shv(\LCHaus_G)}\bigl(\underline{X}, t^*E\bigr).
\]
By the Yoneda lemma, the latter identifies simply with $\bre(X,E)$.

We now turn to the remaining statements and will return afterwards to complete the proof of (1) in full generality. For (2), note that we have an equivalence
\begin{align*}
\mathrm{Pro}(\An^G) &\xrightarrow{\ \simeq\ } \Fun^{\mathrm{Lex}}(\An^G,\An)^\op, \\
Z = ``\lim_i" Z_i &\longmapsto \mathrm{Hom}_{\mathrm{Pro}(\An^G)}(Z,-)
= \colim_i \mathrm{Hom}_{\An^G}(Z_i,-).
\end{align*}
In particular, a pro-anima is completely determined by its corepresented functor which shows (2). 
This equivalence further shows that cofiltered limits and arbitrary colimits of pro-objects correspond to pointwise filtered colimits and limits of the associated corepresented functors (here we use that filtered colimits in anima commute with finite limits). Since $\bre(-,E)$ satisfies cofiltered compact codescent as well as open and closed descent, it follows that $\underline{\Pi}_\infty(-)$ enjoys the properties asserted in (3). The final claim (4) then follows from Proposition \ref{prop_map} using the definition of anima valued singular Bredon cohomology being corepresented (see the beginning of Section \ref{subsec:bredoncw}) and Proposition ~\ref{agreement}.

Finally, using (3), we deduce that for any compactly assembled $D$ and any functor $E\colon \Orb^\op \to D$, the assignment
\[
X \longmapsto C_{\mathrm{Br}}^*(\underline{\Pi}_\infty X,E)
\]
satisfies open descent and cofiltered compact codescent. This follows since for $E : \Orb^\op \to D$ the functor
\[
C_{\mathrm{Br}}^*(-,E): \left(\mathrm{Pro}(\An^G)\right)^\op = \mathrm{Ind}\left((\An^G)^\op\right) \to D 
\]
preserves filtered colimits and arbitrary limits (as it is the ind-extension of a powering and the target is compactly assembled, so limits distribute over colimits). 
The uniqueness theorem (Theorem~\ref{thm:breeq}) then implies that $C_{\mathrm{Br}}^*(\underline{\Pi}_\infty - ,E)$ agrees with $\bre(-,E)$.
\end{proof}

As a result of the last assertion, we see that we can now unleash the full power of equivariant homotopy theory, since this reduces everything to known properties of Bredon cohomology. For example we see that if $E: \Orb^\op \to \Sp$ extends to a spectral Mackey functor (also known as a genuine $G$-spectrum) then also 
$\bre(X, E)$ admits a refinement to a genuine $G$-spectrum $\breg(X,E)$ extending the structure of a naive $G$-spectrum from Remark \ref{rem_genuine}. 
Moreover we get that for $G$-manifolds we have equivariant Poincar\'e-duality etc. 

We now have the following generalization of Corollary \ref{cor_genuine} using again the powering of the dualisable category of motives over $G$-anima (and hence also over pro-$G$-anima).
\begin{coro}\label{cor_genuine_2}
For $X$ a compact $G$-space we have
\[
M_G(\Shv(X, C)) \simeq M_G(C)^{\underline{\Pi}_\infty(X)}
\]
where the power is in the sense of genuine $G$-categories. For $X$ a not-necessarily compact $G$-space we have
\[
M_G(\Shv(X, C)) \simeq M_G(C)^{(\underline{\Pi}_\infty(X^+), \infty)}
\]
where the power is the power of pointed pro-$G$-anima, i.e. the fibre of $M_G(C)^{\underline{\Pi}_\infty(X^+)} \to M_G(C)^{\underline{\Pi}_\infty(\mathrm{pt})} = M_G(C)$.
\qed
\end{coro}

Note that if $X^+$ is sufficiently nice, e.g. itself a $G$-CW complex then we can again write the last power using the singular $G$-anima $\uSing(X)$.

\end{document}